\documentclass{article}
\usepackage{maa-monthly}
\usepackage{subcaption}
\usepackage{hyperref}
\usepackage{arydshln}

\theoremstyle{theorem}
\newtheorem{theorem}{Theorem}
\theoremstyle{proposition}
\newtheorem{proposition}{Proposition}
\newtheorem{conjecture}{Conjecture}
\theoremstyle{lemma}

\theoremstyle{definition}
\newtheorem{definition}{Definition}
\newtheorem*{remark}{Remark}
\usepackage{bm}

\newcommand{\mat}[1]{\ensuremath{\boldsymbol{#1}}}
\newcommand{\matg}[1]{\ensuremath{\bm #1}}
\newcommand{\e}{\ensuremath{\varepsilon}}

\newcommand{\M}{\ensuremath{\mat{M}}}

\begin{document}

\title{A Fractal Eigenvector}
\markright{A Fractal Eigenvector}
\author{Neil J.~Calkin, Eunice Y.~S.~Chan, Robert M.~Corless,\\ David J.~Jeffrey, and Piers W.~Lawrence}
% Calkin, Chan, Corless, Jeffrey, Lawrence
\maketitle

\begin{abstract}
The recursively-constructed family of Mandelbrot matrices $\M_n$ for $n=1$, $2$, $\ldots$ have nonnegative entries (indeed just $0$ and $1$, so each $\M_n$ 
%is 
%what is % RMC Referee 2 objects to "what is"
can be
called a \emph{binary} matrix) and have eigenvalues whose negatives~$-\lambda = c$ give periodic orbits under the Mandelbrot iteration, namely $z_k = z_{k-1}^2+c$ with $z_0=0$, and are thus contained in the Mandelbrot set.  By the Perron--Frobenius theorem, the matrices~$\M_n$ have a dominant real positive eigenvalue, which we call~$\rho_n$.  This article examines the eigenvector belonging to that dominant eigenvalue and its fractal-like structure, and similarly examines (with less success) the dominant singular vectors of $\M_n$ from the singular value decomposition.
\end{abstract}

\noindent

\section{Plots, directed graphs, and an epigraph.}
\begin{quote}
\emph{For a construction to be useful and not mere waste of
mental effort, for it to serve as a stepping-stone to higher
things, it must first of all possess a kind of unity enabling
us to see something more than the juxtaposition
of its elements}.
\hfill---Henri Poincar\'e, \emph{Science and Hypothesis}~\cite{poincare1905science}
\end{quote}

\begin{figure}[h!]
    \centering
    \subcaptionbox{An eigenvector $\mat{u}$}{\includegraphics[width=5.5cm]{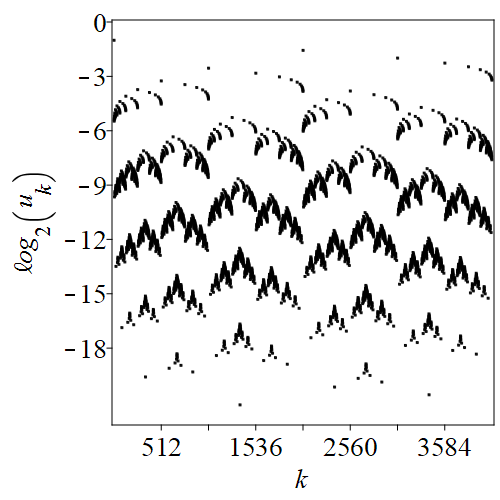}}
    \subcaptionbox{An eigenvector $\mat{v}$}{\includegraphics[width=5.5cm]{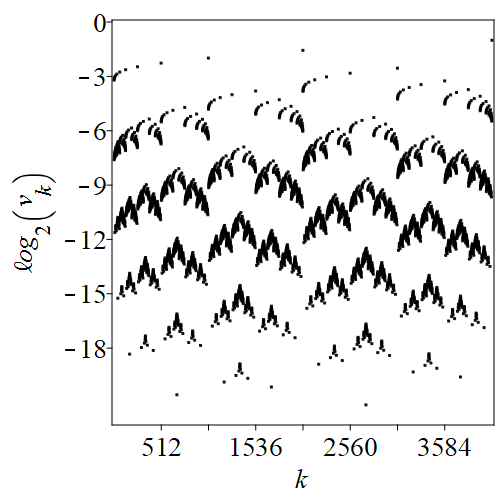}}
    \caption{A discrete plot of the components of eigenvectors corresponding to the (same) dominant eigenvalue of two particular nonnegative integer 
    %(indeed binary, containing only zeros and ones) 
    $4095$-by-$4095$ matrices, graphed on a base-$2$ logarithmic scale.  At this dimension, there are enough components to give the illusion of connected structures, which we seek to understand.  A higher-dimensional plot is shown later, in Figure~\ref{fig:singularvector20}.}
    \label{fig:singularvector12}
\end{figure}

This article seeks to explain a visual curiosity, namely that of Figure~\ref{fig:singularvector12}, where visible structures seem to repeat, slightly transformed, at smaller scales. But what do we mean by an ``explanation?" 
What constitutes a \emph{mathematical} explanation?  By the way, those structures are \emph{not} the result of rounding errors, in spite of our doing the computation only in standard hardware precision floating-point arithmetic.
% RMC Referee 2: lost confidence here!
% Hopefully fixed now.

We will see a connection with the Mandelbrot set.  After seeing the name Mandelbrot get involved, the reader might no longer be surprised that repeating transformed structures occur, because nowadays self-similarity and fractals are familiar features of the mathematical landscape.\footnote{We won't formally define \emph{fractal} here, or pursue the many known facts about the Mandelbrot set. We're going to stick to the finite, and say only that some things that we draw \emph{look like} they might become fractals in the limit as the dimension goes to infinity.}  But can we say more than that, and can we satisfy Poincar\'e's dictum quoted at the start of this section?  We think so.

Let us begin with a recursive  construction of a family of directed graphs (digraphs).  Consider the following digraphs $G_n$ on $d_n=2^{n}-1$ vertices labeled $1, 2, 3, \ldots, d_n$.  For $n=1$ we define the digraph $G_1$ to be just one vertex with one loop; that is, an edge connecting the vertex to itself.  See Figure~\ref{fig:G1}.
\begin{figure}[h!]
    \centering
    \includegraphics[width=4cm]{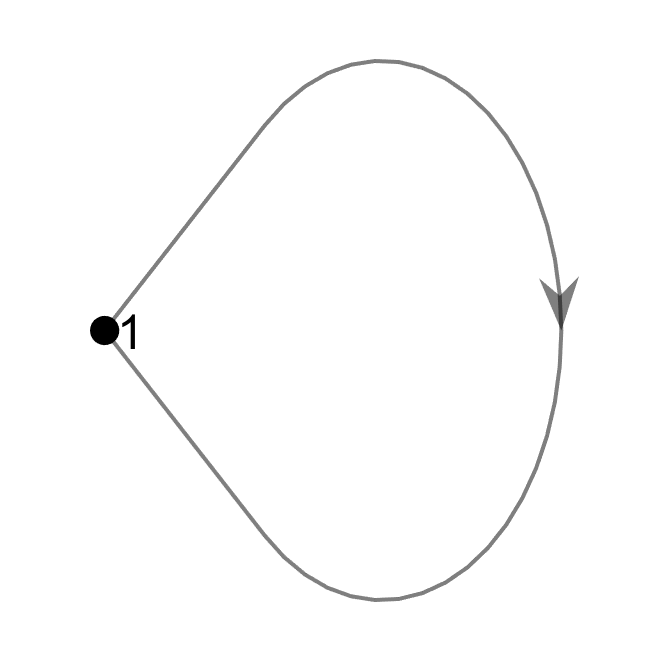}
    \caption{The directed graph $G_1$: just a single loop.}
    \label{fig:G1}
\end{figure}

For $n=2$ we define the digraph $G_2$ to consist of two copies of $G_1$, together with a new vertex between the two copies, and three new edges: two connecting the new vertex to each copy and the third directly connecting the first copy to the second. That is, we make two copies, add a vertex, and connect them all with three new edges. See Figure~\ref{fig:G2}.
\begin{figure}[h!]
    \centering
    \includegraphics[width=5cm]{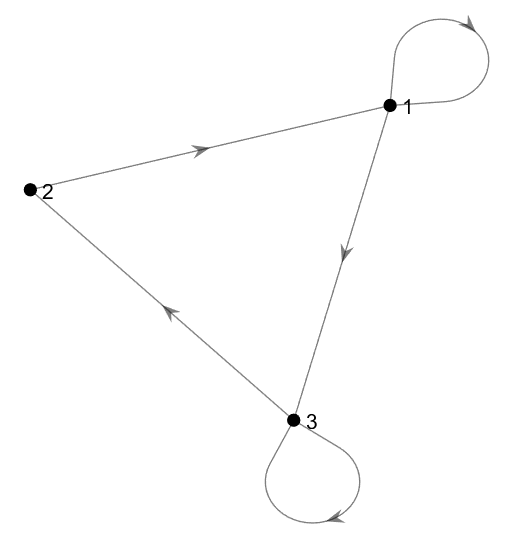}
    \caption{The directed graph $G_2$, drawn with the ``force" model in Matlab: imagine an electrical charge placed at each vertex, repelling all other vertices; and imagine a mechanical spring replacing each edge, pulling connected vertices together.  The pictured configuration is an approximate minimization of the potential energy of this model. At equilibrium in $G_2$, the vertices are equidistant by symmetry. The ``springs" of the two loops have no effect, of course.}
    \label{fig:G2}
\end{figure}

For $n=3$ we repeat the process.  The digraph $G_3$ is defined to consist of two copies of $G_2$, with a new vertex between the two copies and edges connecting the new vertex to each copy and a new edge from the first copy to the second.  Again we have made two copies, added a vertex, and added three edges.  See Figure~\ref{fig:G3}.
\begin{figure}[h!]
    \centering
    \includegraphics[width=6cm]{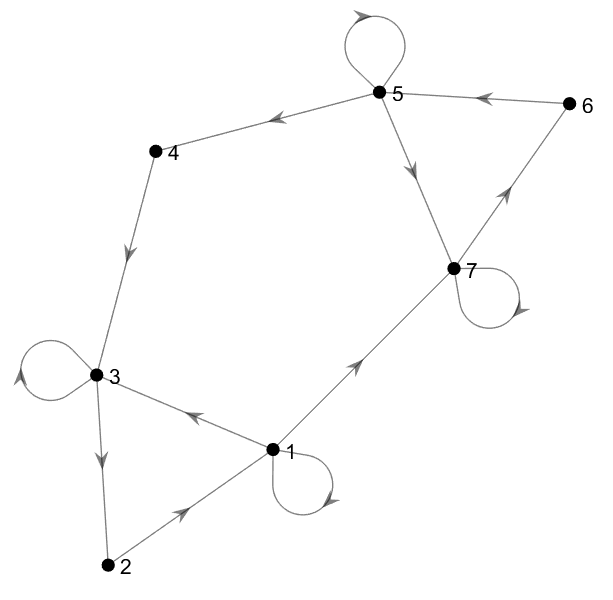}
    \caption{The directed graph $G_3$, containing two copies of $G_2$, drawn with the ``force" model in Matlab.  We see the two copies of $G_2$ with their loops, and the new vertex $4$ between the two copies, and the new edge between vertices $1$ and $7$. We can imagine the effect of the repulsion between vertices owing to the electrical ``charge" being balanced by the pull of the ``springs."}
    \label{fig:G3}
\end{figure}

By now the recursive construction is clear (we will formalize it in Definition~\ref{def:recursivegraph} below), but for thoroughness $G_4$ and $G_5$ are shown in Figure~\ref{fig:G45}.  All of these digraphs are \emph{strongly connected}: that is, there is a \emph{closed walk} along the directed edges that includes all the vertices. Finally, just because we think that the digraphs with lots of vertices are beautiful, we show $G_{12}$ and $G_{13}$ in Figure~\ref{fig:G1213}, with $2^{12}-1 = 4095$ vertices and $2^{13}-1=8191$ vertices, respectively.
\begin{figure}
    \centering
    \subcaptionbox{$G_4$}{\includegraphics[width=6cm]{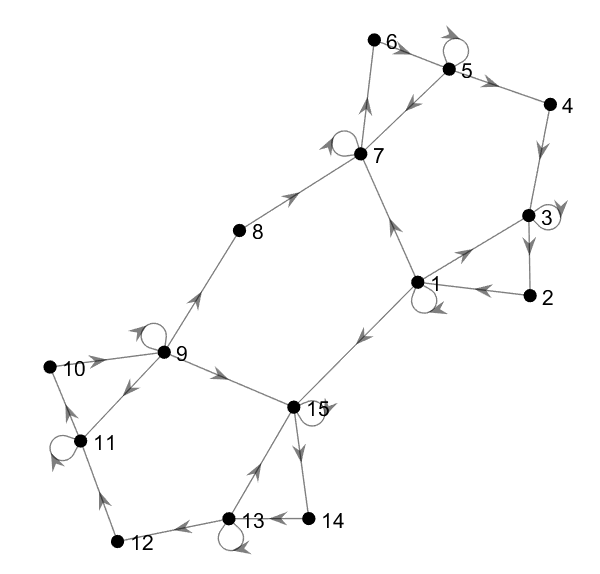}}
    \subcaptionbox{$G_5$}{\includegraphics[width=6cm]{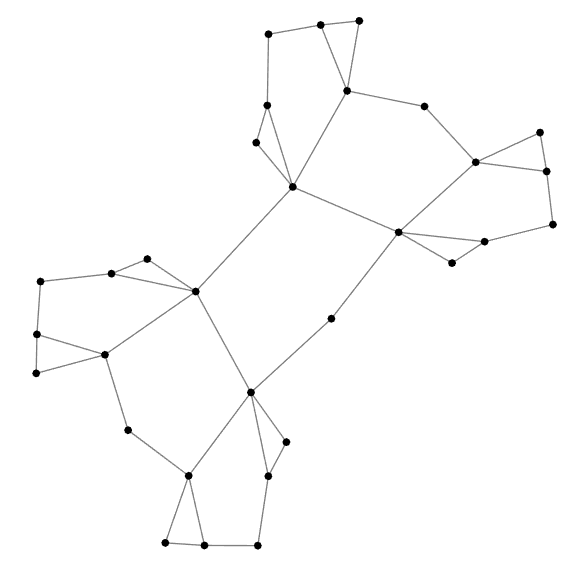}}
    \caption{Force digraphs $G_4$ and $G_5$.  To remove clutter in the larger digraphs, we don't print the arrows or the loops in $G_n$ for $n\ge 5$. The recursive construction should now be clear.}
    \label{fig:G45}
\end{figure}

\begin{figure}
    \centering
    \subcaptionbox{$G_{12}$}{\includegraphics[width=6cm]{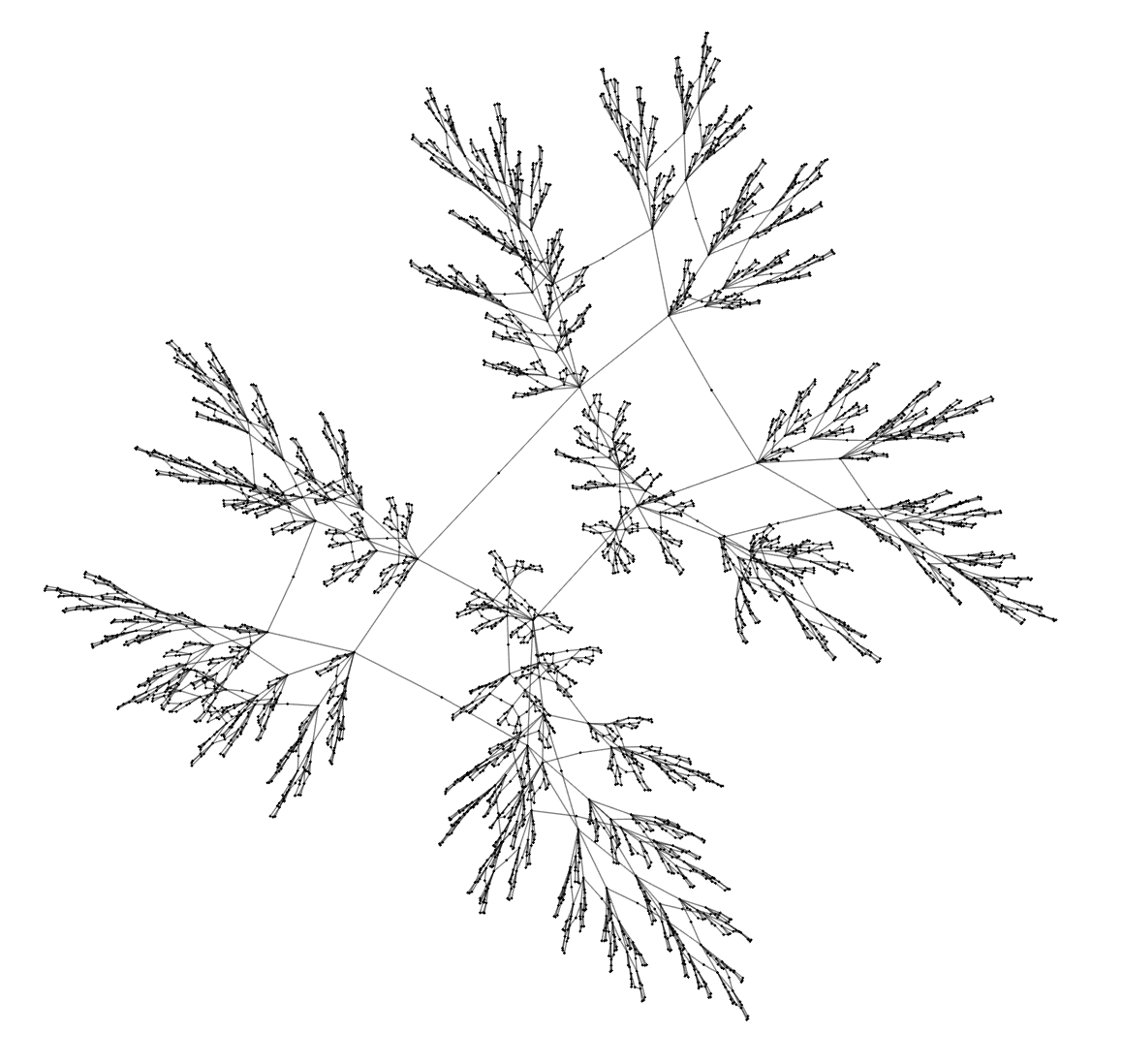}}
    \subcaptionbox{$G_{13}$}{\includegraphics[width=6cm]{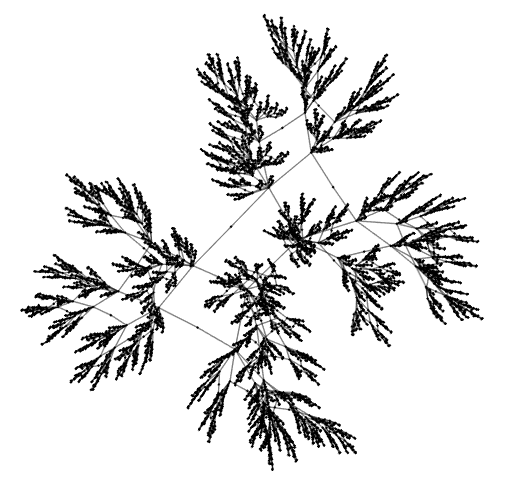}}
    \caption{Force digraphs of $G_{12}$ and $G_{13}$.  The digraph on the right has nearly twice as many vertices and appears darker because of that, but the likeness in shape is evident.}
    \label{fig:G1213}
\end{figure}
\begin{remark}
We use the graph visualization methods which we learned of first  from~\href{https://people.engr.tamu.edu/davis/suitesparse.html}{the SuiteSparse collection of sparse matrices}~\cite{davis2019}; that is, take the graph associated with the matrix and put an attracting spring on every edge, put a repelling charge on every vertex, and look for a minimum energy configuration. 
% See for example~\cite{Ellson03graphvizand} for a  description of this methodology.
We used the implementation in Matlab: for a digraph $G$ issue the command \texttt{plot(G,'Layout','force')}.
\end{remark}
\begin{definition}{}
\label{def:recursivegraph}
$G_1$ is defined as in Figure~\ref{fig:G1}.
%For $n > 1 $ each digraph $G_n$ has vertices labeled 
%$1, 2, \ldots, d_n=2^n-1$.   
To construct $G_n$ for $n>1$, take two copies of $G_{n-1}$ and one new vertex.
Give the number $2^{n-1}$ to the new vertex.  
Keep the numbering on one copy of $G_{n-1}$ the same as it was: $1$ through $2^{n-1}-1$. Renumber the vertices on the other copy to be $2^{n-1}+1$ through $2^n-1$: that is, add $2^{n-1}$ to each vertex number in this copy and renumber all edges $(i,j)$ in this copy of $G_{n-1}$ to become edges $(2^{n-1}+i,2^{n-1}+j)$.  Now add three new edges: the first directed from vertex $1$ to the newly-renumbered vertex $2^n-1$, the second from newly-renumbered vertex $2^{n-1}+1$ to the new vertex $2^{n-1}$, and the third from the new vertex $2^{n-1}$ to vertex $2^{n-1}-1$.
Call the resulting graph $G_n$.
\end{definition}
% RMC Referee 2 did not like the definition
% so I rewrote it.  It is clearer.
\begin{proposition}
$G_n$ is strongly connected.
\end{proposition}
\begin{proof}
By construction, from vertex $1$ we may travel directly to vertex $2^n-1$. Inductively we may travel from vertex $2^n-1$ to $2^{n-1}+1$; by construction from there through vertex $2^{n-1}$ to $2^{n-1}-1$; and inductively from there to vertex $1$.
\end{proof}

\section{The adjacency matrices.}
The \emph{adjacency matrix} $\M$ of a directed graph $G$ with $d$ nodes is a $d$-by-$d$ matrix with entry $M_{i,j}=1$ if there is an edge from vertex $i$ to vertex $j$, and is zero otherwise.
Define $\M_n$ to be the adjacency matrix for $G_n$. For reasons that we will explain soon, we will call them \emph{Mandelbrot matrices}. Explicitly, put 
\begin{equation}
    \M_1 = \begin{bmatrix} 1 \end{bmatrix}\>.
\end{equation}
This is the adjacency matrix for the digraph $G_1$: the matrix contains a $1$ in its $(1,1)$ entry because there is an edge connecting vertex $1$ to itself, i.e., a loop.

We then put 
\begin{equation}
    \M_2 = \begin{bmatrix} \M_1 & 0 & \textcolor{red}{1} \\
    \textcolor{red}{1} & 0 & 0 \\
    0 & \textcolor{red}{1} & \M_1\end{bmatrix}\>.
\end{equation}
% RMC Referee 2 reflected-->situated
This is the adjacency matrix for $G_2$: we have a copy of $G_1$ situated in the upper left corner and another in the lower right corner; we have a new vertex (numbered $2$, in between the copies at $1$ and at $3$) and three new edges (red entries in the matrix) connecting vertex $1$ to vertex $3$, vertex $2$ to vertex $1$, and vertex $3$ to vertex $2$.

Proceeding in a similar fashion to construct $\M_3$, but this time explicitly showing the copies of $\M_2$ in the outlined blocks:
\begin{equation}
    \M_3 = \left[\begin{array}{ccccccc} 
    \cdashline{1-3}
    \multicolumn{1}{:c}{1} & 0 & \multicolumn{1}{c:}{1} & 0 & 0 & 0 & \textcolor{red}{1} \\
    \multicolumn{1}{:c}{1} & 0 & \multicolumn{1}{c:}{0} & 0 & 0 & 0 & 0 \\
    \multicolumn{1}{:c}{0} & 1 & \multicolumn{1}{c:}{1} & 0 & 0 & 0 & 0 \\
    \cdashline{1-3}
    \strut 0 & 0 & \textcolor{red}{1} & 0 & 0 & 0 & 0 \\
    \cdashline{5-7}
    0 & 0 & 0 & \textcolor{red}{1} & \multicolumn{1}{:c}{1} & 0 & \multicolumn{1}{c:}{1}  \\
    0 & 0 & 0 & 0 & \multicolumn{1}{:c}{1} & 0 & \multicolumn{1}{c:}{0} \\
    0 & 0 & 0 & 0 & \multicolumn{1}{:c}{0} & 1 & \multicolumn{1}{c:}{1} \\
    \cdashline{5-7}
    \end{array}\right]\>.
\end{equation}
In general, if $e_1 = [1, 0, 0, \ldots, 0 ]^T$ is the leading elementary column vector of dimension $d_n = 2^n-1$ and $e_{d_n}$ is the final elementary column vector of the same dimension, then we may construct $\M_{n+1}$ from two copies of $\M_n$ in the following way.
% RMC Referee 2 did not like elementary dimension
\begin{definition}{Mandelbrot matrices.}
\label{def:MandelbrotMatrices}
The Mandelbrot matrix $\M_1$ is defined as above, namely the $1$-by-$1$ matrix $\M_1=[1]$.
For $n \ge 1$,
\begin{equation}\label{eq:Mmatrix_recurrence}
    \M_{n+1} = \begin{bmatrix}
    \M_n & \mat{0} & \textcolor{red}{\mat{e_1}\mat{e_{d_n}}^T} \\
    \textcolor{red}{\mat{e_{d_n}}^T} & 0 & \mat{0} \\
    \mat{0} & \textcolor{red}{\mat{e_1}} & \M_n 
    \end{bmatrix}\>.
\end{equation}
\end{definition}
We have the following facts, which we present without proof:
\begin{enumerate}
\item The matrix $\M_{n+1}$ has dimension $d_{n+1} = 2d_n + 1 = 2^{n+1}-1$. 
\item $\det\M_n = 1$ for all $n \ge 1$.
\item $\M_n$ is the adjacency matrix for $G_{n}$ for $n \ge 1$.
\item The matrices $\M_n$ are all 
%what is known as RMC Referee 2 objects to "what is"
``unit upper Hessenberg'': that is, they are upper triangular, except that the principal subdiagonal is also nonzero and contains only $1$s.
\item Since there is a \emph{walk} or directed path in $G_n$ that contains all vertices, i.e.,~a complete circuit, the graph is % what is called RMC Referee 2 objects to "what is" 
\emph{strongly connected} and the adjacency matrices $\M_n$ are 
%what is called RMC Referee 2 objects to "what is"
\emph{irreducible}~\cite[Chapter~40]{hogbenhandbook}.  
%A linear algebraist who does not use graph theory much would know that 
% RMC Referee 2 objects to us imagining
% what a linear algebraist knows or 
% doesn't know.
% One can also deduce that $\M_n$ is irreducible 
% because it is unit upper Hessenberg; the key fact is that the subdiagonal contains no zeros.  
%These two uses of the word convey the same information: the eigenvalue problem for the 
%An irreducible matrix cannot be simply split into two smaller eigenvalue problems.
\item The \emph{period} $h$ of $\M_n$ is defined to be the greatest common divisor (GCD) of the length of all cycles in $G_n$; here this is $h=1$.
\item $\|\M_n\|_1 = \|\M_n\|_\infty = n$.
\item $\|\M_n^{-1}\|_1 = \|\M_n^{-1}\|_\infty = 2n-1$.
\item The number of nonzero entries in $\M_n$ is $2d_n-1$.
\end{enumerate}

The Mandelbrot matrices are defined differently in some works, e.g.,~in~\cite{bini2000design,bini2014solving,CHAN2019373}, so that their characteristic polynomials $p_n(\lambda)$ satisfy $p_0(\lambda)=0$ and the recurrence relation
\begin{equation}
    p_{n+1}(\lambda) = \lambda p_n^2(\lambda) + 1\>. \label{eq:Mandelbrotpolynomials}
\end{equation}
This recurrence relation is a transformation of Mandelbrot's fundamental recurrence $z_{n+1} = z_n^2 + c$; divide that fundamental recurrence by $c$ and put $p_n = z_n/c$ and rename $c$ to be $\lambda$. Zeros of these polynomials (and therefore eigenvalues of the Mandelbrot matrices) give \emph{periodic} points in---centers of hyperbolic components of---the Mandelbrot set.  See Figure~\ref{fig:Mandelbrotgraphs}. These eigenvalues are known to be all \emph{simple}: see~\cite{Morton1994polymap,schleicher2017internal,vivaldi1992ratmap}.

One inconvenience of that alternative definition of the Mandelbrot matrices is that it entails that the entries of each so-defined Mandelbrot matrix are either $0$
or $-1$. In order to minimize minus signs, we changed the definition $\M_n$ so that its entries are either $0$ or $1$---that is, so that $\M_n$ is 
%what is called  RMC Referee 2 objects to "what is"
a \emph{binary matrix}. This has the consequence that the Mandelbrot polynomials as defined above are related to $\det(\lambda \mat{I} + \M_n)$. This \emph{matrix pencil}\footnote{A \emph{matrix pencil} involving the pair of (usually square) matrices $(\mat{A},\mat{B})$ is the linear matrix polynomial $\lambda\mat{B} + \mat{A}$.} has the opposite sign to the usual definition of a characteristic polynomial of a matrix $\mat{A}$, namely $\det(\lambda\mat{I}-\mat{A})$.  

Another difference in our definition here is that the prior definition indexes from $0$. This makes $p_1 = 1$, which has no zeros; this would correspond to the empty matrix,
% RMC going against the editorial which/that
% Here we need "which"!  But we needed a 
% comma, first.
which has no eigenvalues.  Instead, we index from $1$, here.  
In~\cite{corless2013largest} the elements in the matrices are all nonpositive, but the indexing is as here.
%This change encourages off-by-one errors in translation, unfortunately. 
% RMC Referee 2 did not like off-by-one errors
The reason that the other indexing convention is used is so that zeros of $p_n(z)$ give rise to points of period $n$ in the Mandelbrot iteration.  Here, we do not need this, and there are several favorable consequences: for instance, the maximum degree of $G_n$ is $n$, and this means the maximum row sum of $\M_n$ is $n$.  Since the notational ambiguity is already in the literature, we feel required to warn the reader; and we feel entitled to use the most convenient notation here.

\begin{proposition}
The matrices 
$\M_n$ 
as defined 
above satisfy 
$\det(\lambda\mat{I}+\M_n) = p_{n+1}(\lambda)$, where 
$p_n(\lambda)$ 
are the Mandelbrot polynomials defined in equation~\eqref{eq:Mandelbrotpolynomials}.
\end{proposition}
\begin{proof}
%(Proof by Donald E.~Knuth, on a small scrap of paper at SIAM 2016): 
% We'll put that in the acknowledgements
% once the paper has passed review
The determinant function is linear in the entries of the first row.  Therefore the determinant of $\lambda\mat{I} + \M_{n+1}$ is the sum of the determinant of a block lower-triangular matrix with three blocks, namely $\lambda\mat{I}+\M_n$, $\lambda$, and $\lambda\mat{I}+\M_n$ again, and $(-1)^{d_{n+1}-1} = 1$ times the determinant of an upper-triangular matrix  of dimension $d_{n+1}-1$ with ones on the diagonal (because $\M_{n+1}$ is upper Hessenberg with unit subdiagonal). Since the case $n=1$ gives $p_2(\lambda) = \lambda + 1$, the theorem follows by induction. 
\end{proof}
\begin{remark}
To ease reading about and working with these matrices, we define the \emph{characteristic polynomials} $C_n(\lambda) = \det(\lambda \mat{I} - \M_n)$ with the proper signs.  By inspection, we have $C_n(\lambda) = - p_{n+1}(-\lambda)$ because the degrees are always odd: $d_n = 2^n-1$. The recurrence relation that $C_n$ satisfies is  $C_{n+1} = \lambda C_n^2 - 1$ with $C_0=1$.
\end{remark}
\begin{figure}[h!]
    \centering
    \subcaptionbox{$n=6$\label{fig:a}}{\includegraphics[width=5.5cm]{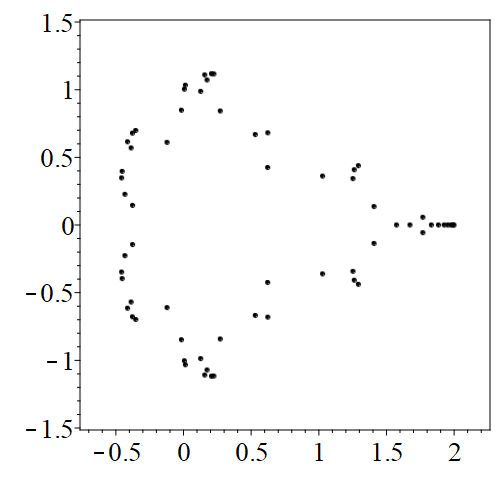}}
    \subcaptionbox{$n=12$\label{fig:b}}{\includegraphics[width=5.5cm]{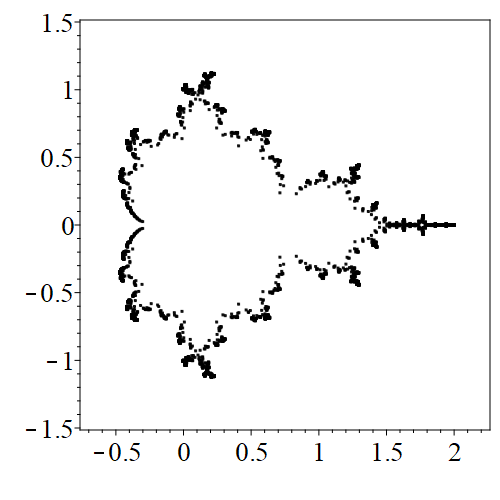}}
    \caption{The eigenvalues of $\M_6$, which are the roots of $C_6(\lambda)$ (Figure~\ref{fig:a}, degree $63$) and the eigenvalues of $\M_{12}$, which are the roots of $C_{12}(\lambda)$ (Figure~\ref{fig:b}, degree $4095$).  The negatives of these points are periodic points in the Mandelbrot set. This is why the name ``Mandelbrot polynomials'' is given to $p_{n+1}(\lambda) = -C_n(-\lambda)$ and why we call the matrices ``Mandelbrot matrices."}
    \label{fig:Mandelbrotgraphs}
\end{figure}

\section{The dominant eigenvalues.}
\begin{quote}
\emph{The eigenvalue of largest absolute value of a positive (square) matrix A is both
simple and positive and belongs to a positive eigenvector. All other eigenvalues are
smaller in absolute value}. \hfill
---O.~Perron, as quoted in~\cite{maccluer2000many}
\end{quote}

A matrix or vector is called \emph{positive} if all its entries are positive.  A matrix is called \emph{nonnegative} if all its nonzero entries are positive. Positive matrices have a dominant eigenvalue $\rho$ which is also positive, by the Perron--Frobenius theorem. 
% RMC Referee 2 didn't like clearly a limit
A nonnegative matrix can be taken as a limit of a set of positive matrices and therefore its eigenvalue of largest absolute value, $\rho$, is also nonnegative and the eigenvector belonging to it is nonnegative; however, there may be other eigenvalues of equal (largest) magnitude.

In the case in which the matrices are 
%what is called RMC Referee 2 objects to "what is" 
\emph{irreducible}, more can be said.  In such a case, the \emph{period} $h$ of the matrix is defined to be the GCD
of the lengths of all the circuits in the digraph associated with the matrix.  Then those other equally largest magnitude eigenvalues must be of the form $\exp(2\pi i/h)\rho$.

Since our matrices $\M_n$ are nonnegative, and since the digraphs $G_n$ associated with the matrices are strongly connected (which implies the matrices $\M_n$ are irreducible) the Perron--Frobenius theory applies.  Since there are cycles of length $1$, we see that the period $h$ as defined in Definition~\ref{def:MandelbrotMatrices} is just $1$ and therefore the largest eigenvalue is in fact unique.  See~\cite{maccluer2000many} for several proofs of Perron's theorem, which states that an irreducible nonnegative matrix has a single, simple, positive, largest real eigenvalue, denoted $\rho$.  %This dominant eigenvalue is often denoted $\rho$, because this symbol often denotes the \emph{spectral radius}.  
The dominant eigenvalue of $\M_n$, which we will call $\rho_n$, has been found in~\cite{corless2013largest} to have the asymptotic expansion, valid as $n\to \infty$,
\begin{equation}
    \rho_n = 2 - \frac{3}{8}\pi^2 4^{-n} + \widetilde{O}(4^{-2n})\>.\label{eq:asymptrho}
\end{equation}
Here the ``soft-Oh'' notation $\widetilde{O}(g(n))$ is shorthand for $O(g(n)\log^m g(n))$ for some fixed $m$.  For instance, $n\cdot 4^{-2n}$ and $n^{10} 4^{-2n}$ are both $\widetilde{O}(4^{-2n})$.
% RMC Referee 2 didn't like soft O
In this article we are 
concerned not with the eigenvalue, but rather
with the eigenvector belonging to it. For our purposes, a more accurate $\rho_n$ can be found by simple Newton iteration on the recurrence relation\footnote{As detailed in~\cite{chan2016comparison} the coefficients of the expanded polynomial grow \emph{doubly exponentially} with $n$, and it is a bad idea numerically to do that expansion before trying to find roots.  As a further benefit, using the recurrence relation instead takes only $O(n)$ operations, whereas evaluating the polynomial any standard way would require $O(d_n)$ operations, i.e., exponentially greater cost.  There are fewer rounding errors, too.} $C_{k+1}(z) = z C_k^2(z)-1$ with $C_0(z)=1$, so of course $C_{k+1}'(z) = C_k^2(z) + 2zC_k(z)C_k'(z)$ can be computed simultaneously. It is interesting to note as the authors of~\cite{corless2013largest} do that, because the derivatives of $C_k(z)$ are so large, starting with just $\rho_n \approx 2$ is \emph{not good enough for convergence}, and one must use the starting estimate from equation~\eqref{eq:asymptrho}.
For instance, for $n=7$ the asymptotic estimate gives $\rho_7 \doteq 1.999774\textcolor{red}{10268247}$, and two Newton iterations achieve full double-precision accuracy at $\rho_7=1.99977404869373$; comparison shows the red digits were wrong.
% RMC Referee 2 thought we might be so
% silly as to use the monomial basis here.

This treatment works perfectly, though, and \emph{we may regard the dominant eigenvalue as known to full accuracy}.  Computing the eigenvalue is usually the hard part, but not here.  We may now continue with the \emph{eigenvector}.

The Perron--Frobenius theory states that each entry of the eigenvector belonging to $\rho_n$ can \emph{also} be taken to be nonnegative. We will see that, in practice, all entries are in fact positive.

To compute the eigenvector once the eigenvalue $\rho_n$ is known, we do the simplest thing imaginable: we put $x_{d_n}=1$ and $\hat{\mat{x}} = [x_1, x_2, \ldots, x_{d_n-1}]^T$ and solve the (very sparse) triangular system $\mat{T}_n\hat{\mat{x}} = -x_{d_n}\mat{b} $ that arises from looking for the null vector of $\mat{R}(\rho_n) = (\M_n - \rho_n\mat{I})$. Delete the first row of $\mat{R}$ and call the result $\widetilde{\mat{R}}$. The vector $\mat{b}$ is the last column of $\widetilde{\mat{R}}$ and the upper triangular matrix $\mat{T}_n$ comprises the first $d_n-1$ columns of $\widetilde{\mat{R}}$.

For the \emph{eigenvector} problem, the accurate computation of the dominant eigenvector of $\M_n$ therefore costs only $O(d_n)$ arithmetic operations.\footnote{At precision higher than offered by hardware, the bit complexity becomes relevant---not just the number of operations themselves---because the cost of each arithmetic operation increases if the precision is increased.
Since the numerical condition of a generic eigenvalue problem is expected to grow like $O(d_n^2)$, and does so in this case, one expects to have to use greater than double precision if $d_n > 10^8$, which occurs if $n > 26$.
} For a dense upper Hessenberg matrix the cost would instead be $O(d_n^2)$. 
Why is the computation so cheap?
Basically, because the matrix $\M_n$ is so sparse (it has only $2d_n-1$ nonzero entries). 

% Computation of the dominant eigenvalue itself is also inexpensive, because we know an asymptotic expansion for it, and we can improve that estimate using Newton's method on the recurrence relation $C_n(z) = zC_{n-1}^2(z)-1$, which because it only involves $O(n)$ operations, i.e., $O(\log d_n)$ operations, is both inexpensive and numerically stable.  
Because the recurrence relation for the characteristic polynomial is so economical, this technique is cheaper than the more general technique for \emph{quasiseparable} matrices discussed in~\cite{Eidelman2005}, which also uses Newton's method on the characteristic polynomial.
% We believe that Mandelbrot matrices are ``$(1,n/2)$-quasiseparable" in the terminology of that paper and so their algorithm could be used on the full eigenproblem; of course we are concerned here with only one eigenvalue, which we already know well.

\section{``Fractal'' eigenvectors.}
We begin with pictorial representations of these eigenvectors, by plotting the components $x_k$ against their index, $k$.  The eigenvector of $\M_1$ is trivial, being just a single dot: when $k=1$, $x_k = 1$.  This needs little comment.  So let us consider instead $\M_2$.  We choose to normalize the eigenvector by taking $x_d = 1$, and denote it as $[x_1, x_2, 1]^T$. We then have
\begin{equation}
    \begin{bmatrix}
    1 & 0 & 1 \\
    1 & 0 & 0 \\
    0 & 1 & 1
    \end{bmatrix}\begin{bmatrix}
    x_1 \\
    x_2 \\
    1
    \end{bmatrix}
    = \rho_2 \begin{bmatrix}
    x_1 \\
    x_2 \\
    1
    \end{bmatrix}\>.
\end{equation}
This gives $x_2 = \rho_2-1$ from the third equation and $x_1 = \rho_2(\rho_2-1)$ from the second; the remaining equation simply gives (of course) the characteristic polynomial that $\lambda=\rho_2$ has to satisfy, namely $C_2(\lambda) = \lambda^3 - 2\lambda^2 + \lambda - 1 = 0$.  We plot this eigenvector in Figure~\ref{fig:logeigenvector2}.

% For definiteness, we compute $\rho_2$ exactly: to write it neatly, put $\alpha = (100+12\sqrt{69})^{1/3}$ and then $\rho_2 = 2(1+1/\alpha)/3 + \alpha/6 \approx 1.7549$.
% Incidentally, the asymptotic formula from equation~\eqref{eq:asymptrho} gives $1.76868\ldots$, which shows that the asymptotic formula is already reasonably accurate at $n=2$. 

It will turn out to be convenient to normalize these eigenvectors by $x_d = 1$ in analytic computation; however, for visual presentation when there are many components, the plots turn out to be more intelligible if instead we choose $x_1 = 1$.  We do this in all of Figures~\ref{fig:logeigenvector2}--\ref{fig:logeigenvector1314}. This means there is always a component plotted in the upper left corner.
\begin{figure}
    \centering
    \includegraphics[width=5cm]{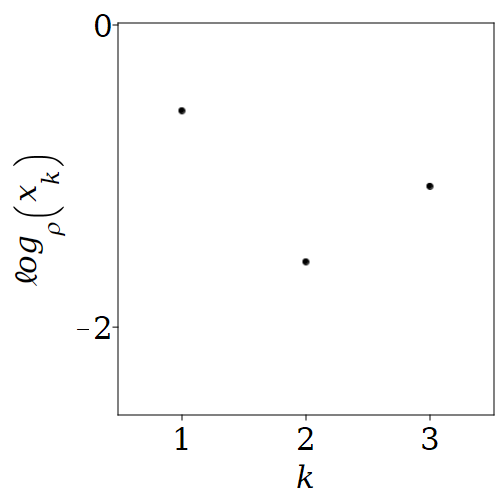}
    \caption{A semilog plot of the eigenvector components in the case $n=2$ (dimension $d_n=2^n-1=3$) of the dominant eigenvector of $\M_n$ (normalized to have $x_1=1$). At this dimension it does not seem useful to make a discrete plot of the components of the eigenvector.  Such plots \emph{are} made in vibration studies, where the eigenvector gives the so-called ``mode shape.''  
    %RMC Referee 2 objects to "what is"
    Here the purpose will only become clear as we increase the dimension.}
    \label{fig:logeigenvector2}
\end{figure}
\begin{figure}
    \centering
    \subcaptionbox{$n=3$}{\includegraphics[width=5cm]{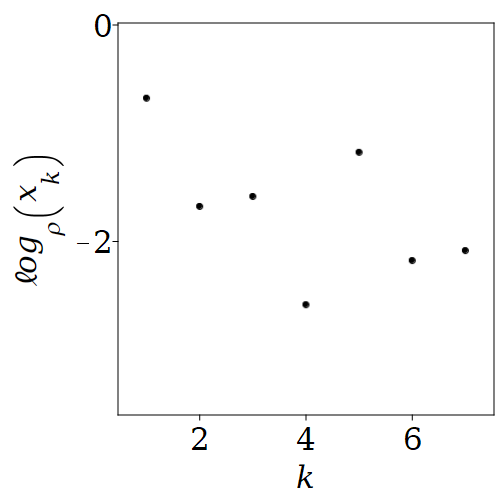}}
    \subcaptionbox{$n=4$}{\includegraphics[width=5cm]{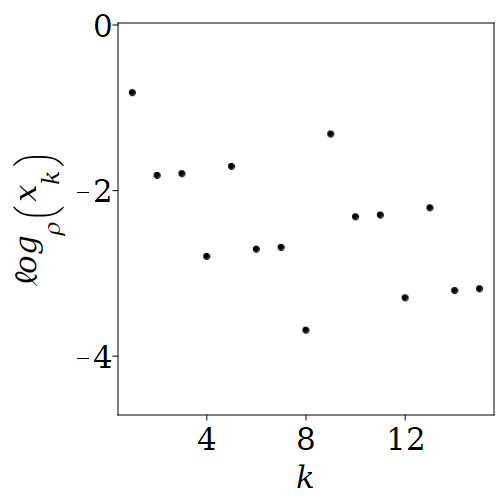}}
    \caption{A semilog plot of the eigenvector components in the cases
    $n=3$ (dimension $d_n=2^n-1=7$) and $n=4$ (dimension $d_n=2^n-1=15$) of the dominant eigenvector of $\M_n$. The largest component is $x_1 = 1$ (upper left corner). Notice the curious symmetry of the final half of the eigenvector components compared to the first half: ignoring the component exactly in the middle, the first half is a copy of the second, but scaled upwards slightly. }
    \label{fig:logeigenvector34}
\end{figure}
\begin{figure}
    \centering
    \subcaptionbox{$n=13$}{\includegraphics[width=5cm]{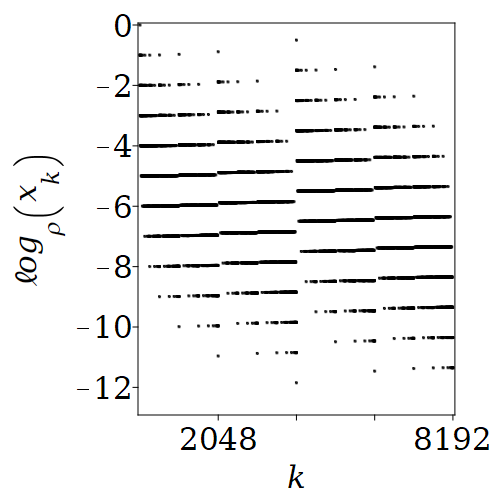}}
    \subcaptionbox{$n=14$}{\includegraphics[width=5cm]{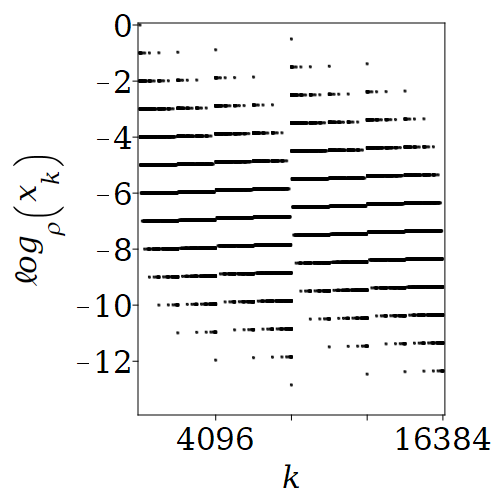}}
    \caption{A semilog plot of the eigenvector components in the cases
    $n=13$ (dimension $d_n=2^n-1=8191$) and $n=14$ (dimension $d_n=2^n-1=16383$) of the dominant eigenvector of $\M_n$. In these figures we normalized so that the largest component is $x_1=1$.  }
    \label{fig:logeigenvector1314}
\end{figure}

\subsection{Maybe the simplest explanation.}
When we look at Figures~\ref{fig:logeigenvector34}--\ref{fig:logeigenvector1314} we see that each eigenvector can be split into (nearly) two halves:  $d_n = 2^n-1$ is odd and so the middle component can be taken to be special. The $2^{n-1}$-dimensional subvectors consisting of the two halves of the remaining elements have a symmetry that, once seen, is striking: the two halves are visually identical, apart from scaling.
Indeed we will prove that there is a single common scaling factor relating the two halves: $x_j = K x_{j+2^{n-2}+1}$ for $1 \le j \le 2^{n-2}$.

Also, there seems to be a significant likeness of the second half of the eigenvector to the full eigenvector of the previous case ($n-1$).  What could explain that?

An important element of the explanation comes from the following observation. Suppose $x(\rho_n)$ (normalized so its \emph{final} entry is $1$) is the eigenvector of $\M_n$ belonging to $\rho_n$.
Each component of $x(\rho_n)$ is a polynomial in $\rho_n$. For instance, when $n=2$ we have
\begin{equation}
    x = \left[ \begin {array}{c} \rho_{{2}} \left( \rho_{{2}}-1 \right) 
\\ \noalign{\medskip}\rho_{{2}}-1\\ \noalign{\medskip}1\end {array}
 \right] \>,
\end{equation}
where $C_2(\rho_2) = \rho_2^3 - 2\rho_2^2+\rho_2 -1 = 0$.
For $n=3$ we have instead (after expanding and factoring the results symbolically\footnote{This is not a sensible thing to do numerically; these explicit expressions for eigenvector components rapidly become numerically unstable, owing to the \emph{doubly} exponential growth of the monomial basis coefficients~\cite{corless2013largest}. We do not use these symbolic expressions for numerical computation.})
% RMC Referee 2 wanted the warning added
\begin{equation}
x =  \left[ \begin {array}{c} {\rho_{{3}}}^{2} \left( \rho_3-1
 \right)  \left( {\rho_{{3}}}^{3}-2\,{\rho_{{3}}}^{2}+\rho_{{3}}-1
 \right) \\ 
 \noalign{\medskip}\rho_{{3}} \left( \rho_3-1 \right) 
 \left( {\rho_{{3}}}^{3}-2\,{\rho_{{3}}}^{2}+\rho_{{3}}-1 \right) 
\\ 
\noalign{\medskip} \rho_{{3}}\left( {\rho_{{3}}}^{3}-2\,{\rho_{{3}}}^{2}+\rho
_{{3}}-1 \right) \\ \noalign{\medskip}{\rho_{{3}}}^{3}-2\,{
\rho_{{3}}}^{2}+\rho_{{3}}-1\\ \noalign{\medskip}\rho_{{3}} \left(
\rho_{{3}}-1 \right) \\ \noalign{\medskip}\rho_3-1
\\ \noalign{\medskip}1\end {array} \right]\>.
\end{equation}
Notice the occurence of $C_2(\rho_3)$ in this vector.  Because $C_3(\rho_3) = \rho_3C_2^2(\rho_3)-1 = 0$, we may write this as $C_2(\rho_3) = 1/\sqrt{\rho_3}$.
Notice also that the final three components are the \emph{same} polynomials as occurred for $n=2$, only now evaluated at $\rho_3$, not $\rho_2$.

This is because
$\M_{n+1}$ has two copies of $\M_n$ in it, and is upper Hessenberg so that we may find the eigenvector by solving a unit upper triangular system. In block form, we have
\begin{equation}
    \begin{bmatrix}
    \M_n & \mat{0} & \mat{e}_1\mat{e}_{d_n}^T \\
    \mat{e}_{d_n}^T & 0 & \mat{0} \\
    \mat{0} & \mat{e}_1 & \M_n
    \end{bmatrix}
    \begin{bmatrix}
    \mat{\Tilde{x}}\\
    u\\
    \mat{x}
    \end{bmatrix} = \rho_{n+1}    \begin{bmatrix}
    \mat{\Tilde{x}}\\
    u\\
    \mat{x}
    \end{bmatrix}\>.\label{eq:blockform}
\end{equation}
\begin{theorem}\label{thm:blockform}
The solution to equation~\eqref{eq:blockform} can be constructed recursively as follows. 
Put $\mat{x}_1(\rho) = [1]$, a one-vector containing a trivial polynomial in $\rho$. Subsequent vectors of dimension $2^{n+1}-1$ are defined by the following polynomial vector recurrence relation: 
\begin{equation}
    \mat{x}_{n+1}(\rho_{n+1}) = \begin{bmatrix}
    \rho_{n+1}C_n(\rho_{n+1})\mat{x}_n(\rho_{n+1}) \\
    C_n(\rho_{n+1})\\
    \mat{x}_n(\rho_{n+1})
    \end{bmatrix}\>. \label{eq:blocksolution}
\end{equation}
\end{theorem}
\begin{proof}
Notice first that the final component of each $\mat{x}_n(\rho_{n+1})$ is $1$, as intended.
Substitution of equation~\eqref{eq:blocksolution} into equation~\eqref{eq:blockform} gives two matrix equations, \eqref{eq:first} and~\eqref{eq:third} and a scalar equation, \eqref{eq:second}:
\begin{equation}
    \rho_{n+1}C_n(\rho_{n+1})\M_n\mat{x}_n(\rho_{n+1}) + \mat{e}_1 = \rho_{n+1}^2C_n(\rho_{n+1})\mat{x}_n(\rho_{n+1})\>,\label{eq:first}
\end{equation}
\begin{equation}
    C_n(\rho_{n+1})\mat{e}_1 + \M_n \mat{x}_n = \rho_{n+1}\mat{x}_n(\rho_{n+1})\>,\label{eq:third}
\end{equation}
and, because the final component of $\mat{x}_n(\rho_{n+1})$ is $1$, the scalar equation simply becomes the identity
\begin{equation}
    \rho_{n+1}C_n(\rho_{n+1}) = \rho_{n+1}C_n(\rho_{n+1})\>.\label{eq:second}
\end{equation}
Next, using $C_n(\rho_{n+1}) = 1/\sqrt{\rho_{n+1}}$, we see that the matrix equation~\eqref{eq:first} is a simple scalar multiple of the matrix equation~\eqref{eq:third}.  Thus, we only need to solve this final matrix equation.  But this has already been done, recursively: $\M_n$ is upper Hessenberg, so the eigenvector $\mat{x}_n(\rho_{n+1})$ is completely determined by solving rows $2$ through $d_n$ by back substitution given that the final component is $1$.  
\end{proof}
\begin{remark}
The unused row in the matrix equation, namely 
$$
C_n(\rho_{n+1}) + \sum_{j\ge1} M_{1,j}x_j(\rho_{n+1}) = \rho_{n+1}x_1(\rho_{n+1})\>,
$$ must simply be a restatement of the characteristic polynomial.  In some sense we don't need to explicitly solve it: we know how it will work out because by definition $\mat{x}_{n+1}$ is an eigenvector and the only variable left free is $\rho_{n+1}$.  Nonetheless, it is an interesting equation to solve: it involves the $1, 3, 7, \ldots, 2^n-1$ components of $\mat{x}_n(\rho_{n+1})$ (these are the only entries of the first row of $\M_n$ that are nonzero) and does not lead directly to the recurrence relation $C_{n+1}(\rho) = \rho C_n^2(\rho)-1$ but rather needs to use it and the recursive construction of the vector $\mat{x}$ itself.  We leave this as fun for the reader, but note that it gives a sparse representation for $C_{n+1}(\rho)$ that may have other uses.
\end{remark}
\begin{remark}
The details of that proof also identify both the smallest element of that vector, namely $C_n(\rho_{n+1}) = 1/\sqrt{\rho_{n+1}}$ in the middle, and the largest entry $x_1$, which is $\sqrt{\rho_{n+1}}$ times an approximation for the largest entry of the previous vector. \end{remark}

% To compare the lower half of $\mat{x}_{n+1}$ to the previous vector $\mat{x}_n$, let us examine the middle entry of that lower half, which is $C_{n-1}(\rho_{n+1})$, with the corresponding entry $C_{n-1}(\rho_n)$ in $\mat{x}_n$.  We have $C_{n-1}(\rho_{n-1})=0$, and for all $\rho \in (\rho_{n+1},2)$ we have $0 < C_{n-1}(\rho) < 1$ (establishing the upper limit is an easy induction).  Since $\rho_{n-1} < \rho_n < \rho_{n+1} < 2$ this means that $0 < C_{n-1}(\rho_n) = 1/\sqrt{\rho_n}$, justifying our choice of sign of the square root earlier. By using the recurrence twice we can see that
% \[
% C_{n-1}(\rho_{n+1}) = \sqrt{\frac{1 + \rho_{n+1}^{-1/2}}{\rho_{n+1}}}\>.
% \]

Since equation~\eqref{eq:blocksolution} shows that the lower entries of $x_{k}(\rho_{k+1})$ are \emph{fixed} polynomials in $\rho_{k+1}$, and we know $\rho_k \to 2$ as $k \to \infty$, those lower entries actually converge to $[1, 1, 2, 1, 2, 2, 4, 1, 2, 2, 4, 2, 4, 4, 8, 1, \ldots ]$.  This shows up as \href{http://oeis.org/A048896}{Sequence A048896 in the On-line Encyclopedia of Integer Sequences} and is connected to Catalan numbers and to the number of $1$s in the binary expansion of $n$, apparently~\cite{sloane2021}. We can see in retrospect that this is natural: each entry of the eigenvector is either a power of $\lambda$ at $\rho_{n+1}$ or an evaluation of some $C_k(\lambda)$ at $\rho_{n+1}$, and these go to $2$ or $1$, respectively, as $k \to \infty$.
We do not pursue this further here, although it is extremely tempting.

The \emph{upper} part of the vector is somehow more surprising: the leading entry is
\begin{equation}
    x_{n+1,1}(\rho_{n+1}) = \rho_{n+1}^{n} \prod_{k=1}^{n} C_k(\rho_{n+1}) = \frac{2^{n+1}}{\pi}\left(1 + \widetilde{O}(4^{-n})\right) \>.\label{eq:conjecturepi}
\end{equation}
Note that $\rho_{n+1}C_k(\rho_{n+1})$ for $k=1$, $\ldots$,  $n$ are the nonzero elements of the generated periodic orbit of the Mandelbrot set.
We have established that last asymptotic equality only by high-precision numerical experiments, up to $n = 15$ where $d_{n+1} = 65,535$.
We are quite convinced it's true, but have no proof.
We do not pursue this further here either, although it is also extremely tempting.
Another interesting and unexplained experimental fact is that the top of the vector, once the factor of $\pi$ has been removed, appears to be in a scaled \emph{Gould's sequence} \href{http://oeis.org/A001316}{oeis.org/A001316}: if we compute $x_{16}(\rho_{16})$ and scale the topmost $16$ entries (say), we get
$2^{-11}\pi x_{16,16:1} = [ 1, 2, 2, 4, 2, 4, 4, 8, 2, 4, 4, 8, 4, 8, 8, 16]$.  Gould's sequence is visible at least up to the topmost $128$ entries.

% 1, 1, 2, 1, 2, 2, 4, 1, 2, 2, 4, 2, 4, 4, 8, 1, 2, 2, 4, 2, 4, 4, 8, 2, 4, 4, 8, 4, 8, 8, 16, 1

% The eigenvector solution process involves an upper triangular matrix whose condition number we asked you to show was ``only" $O(2^{n})$, and we conclude that changing the matrix by $O(4^{-n})$ will change the results only by at most $O(2^{-n})$.  Indeed this is an upper bound for the differences that we see between the half of the eigenvector $\mat{x}_{n+1}$ and the whole $\mat{x}_n$, and so we believe that we have ``explained'' this near self-similarity.

The recursive application of powers of $\rho_j$, all nearly equal to $2$, explains the bands visible on a $\log_2$ scale.
Since the upper half of the eigenvector is a scaled version of the lower half, with the same scaling factor $\sqrt{\rho_{n+1}}$ applied to each component, this explains the rest. For \emph{this} question, we believe that this answer satisfies Poincar\'e's dictum because in order to reach our explanation, we had to use several powerful mathematical ideas. 
We now turn to a harder problem.

\section{Singular values and vectors.}
Matlab's sparse singular value decomposition (SVD) routine for $\mat{A} = \mat{U}\mat{\Sigma}\mat{V}^T$ can compute the singular vectors belonging to the largest singular value of $\M_n$ quite rapidly---seemingly also of cost $O(d_n)$--- and moreover to do so accurately.  As an instance of timing, Matlab 2019b can compute and plot each of the dominant left and right singular vectors for $n=20$, which means $d_{20} = 2^{20}-1 = 1,048,575$, in under 11 seconds on a 2017 Microsoft Surface Pro; that is, it can work with a (very sparse, true) million-plus by million-plus matrix and compute two million-plus vectors in a ludicrously short time, on a tablet computer.  
% Similarly, Maple's SVD routine, which can be called with variable precision in order to compute small singular values to high relative accuracy, is also rapid enough to produce useful results.

In order to \emph{explain} the features of Figure~\ref{fig:singularvector20} we are going to have to use some facts about the singular value decomposition. The following section summarizes some things we need.

% The largest singular values of the first few $\M_n$ are, to five figures,
% $1.0000$ for $n=1$, $1.8019$ for $n=2$, $2.2948$, $2.6933$, $3.0385$, $3.3478$, $3.6308$, $3.8934$, $4.1394$,
% $4.3717$, $4.5924$, and $4.8030$ for $n=12$. We do not immediately see any regularity in that list. Continuing, we see slow steady growth.  By $n=20$ the largest singular value is about $6.2393$. We have a simple bound (from the degree) and a better \emph{conjectured} bound for these, in equation~\eqref{eq:conjecturedbound} which we will see in a moment.

% In contrast, the largest eigenvalue for $\M_{12}$ is $2 - 3\pi^2/(8\cdot 4^{12}) \approx 1.999999779397$; this regularity in the eigenvalue case is a significant help. [There are ten eigenvalues of $\M_{12}$ bigger than $1.99992$, by the way; the dominant eigenvalue isn't \emph{very} dominant.]
\begin{figure}[t]
    \centering
    %\subcaptionbox{Singular vector $\mat{u}$}
    {\includegraphics[width=10cm]{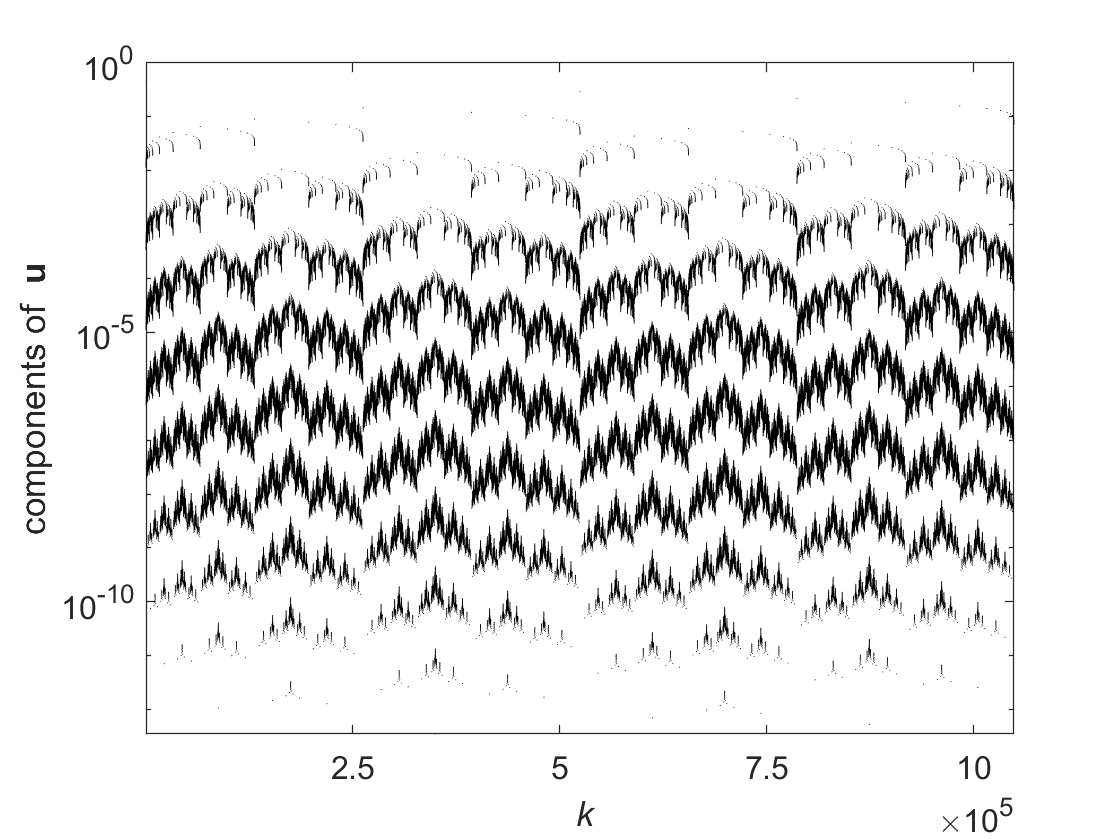}
    }
    %\subcaptionbox{Singular vector $\mat{v}$}{\includegraphics[width=6cm]{singularvectorv20.png}}
    \caption{A discrete plot of the components of the singular vector (dimension $2^{20}-1 = 1,048,575$) corresponding to the dominant singular value of $\M_{20}$ as computed by Matlab with its sparse SVD command \texttt{svds}, plotted on a logarithmic scale. We see many complex structures.}
    \label{fig:singularvector20}
\end{figure}

\begin{figure}
    \centering
    \includegraphics[width=8cm]{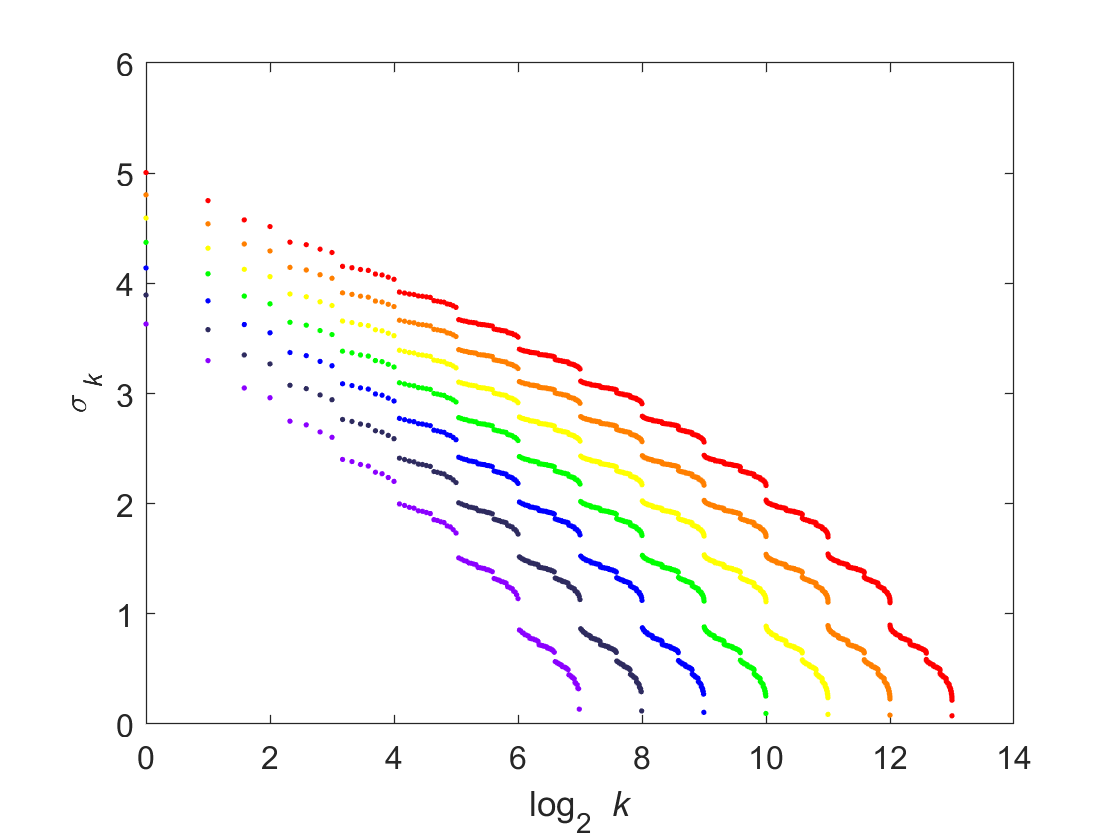}
    \caption{All singular values $\sigma_{n,k}$ for $1 \le n \le 2^{k}-1$ of $\M_k$ for various $k$; specifically, $\M_{7}$ (in violet), $\M_{8}$ (in indigo), $\M_9$ (in blue), $\M_{10}$ (in green), $\M_{11}$ (in yellow), $\M_{12}$ (in orange), and $\M_{13}$ (in red), plotted against $\log_2 n$.  In this plot we see apparent self-similarity at different scales: for each increment in $k$, the singular values move up and a likeness of the old ones seems to be inserted at the bottom. }
    \label{fig:singularvaluesdistribution}
\end{figure}

\subsection{The Jordan--Wielandt matrix.} The singular values of $\M_n$ can be found from the eigenvalues of the well-known Jordan--Wielandt matrix corresponding to $\M_n$:
\begin{equation}
    \begin{bmatrix}
    0 & \M_n \\
    \M_n^T & 0
    \end{bmatrix}\>.\label{eq:JordanWielandt}
\end{equation}
See Figure~\ref{fig:JW34} for some digraphs associated with these matrices.
It is important to note that these graphs are \emph{bipartite}: we can divide the vertices into two groups (say ``red'' and ``green'') and each vertex is connected only to vertices of the other color.
\begin{figure}
    \centering
    \subcaptionbox{Jordan--Wielandt $n=3$}{\includegraphics[width=5.5cm]{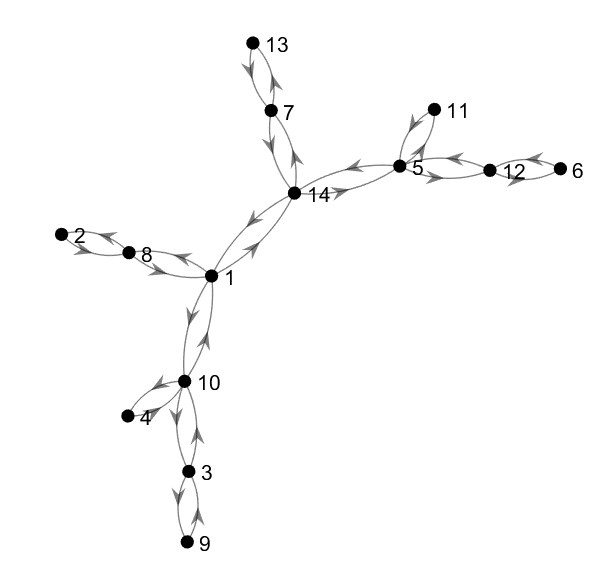}}
    \subcaptionbox{Jordan--Wielandt $n=4$}{\includegraphics[width=5.5cm]{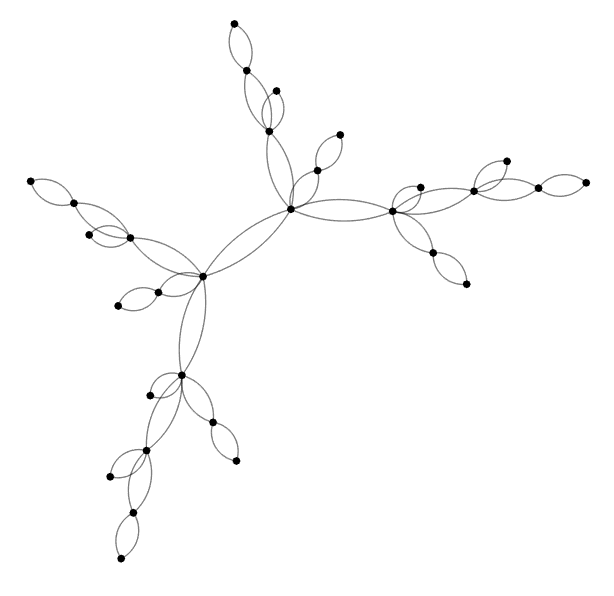}}
    \caption{Force digraphs associated with the Jordan--Wielandt matrix of equation~\eqref{eq:JordanWielandt}. The period ($2$) can be seen in these bipartite digraphs. As before, we declutter the digraph if the number of vertices is too high to show labels and arrows effectively. }
    \label{fig:JW34}
\end{figure}

Alternatively, we could use the eigenvalues of $\M_n^T\M_n$ and $\M_n\M_n^T$ which give the \emph{squares} of the singular values of $\M_n$.
But let us continue with the Jordan--Wielandt matrix.

If the singular value decomposition of $\M_n$ is given by $\M_n = \mat{U}\matg{\Sigma}\mat{V}^T$, with orthogonal matrices $\mat{U}$ and $\mat{V}$ and diagonal matrix $\matg{\Sigma}$ with its entries ordered\footnote{Here we have a notational conflict.  We would like to use the notation $\sigma_k$ to refer to the \emph{largest} singular value of the matrix $\M_k$, but this is confusing; ordinarily the largest singular value of a matrix is $\sigma_1$. We will use $\sigma_{1,k}$ to mean the largest singular value of $\M_k$.} so that $\sigma_1 \ge \sigma_2 \ge \cdots$, then we can form an invertible matrix
\begin{equation}
    \mat{X} = \begin{bmatrix}
    \mat{U} & -\mat{U}\\
    \mat{V} & \mat{V}
    \end{bmatrix}\>,
\qquad{}
    \mat{X}^{-1} = \frac12 \begin{bmatrix}
    \mat{U}^T & \mat{V}^T \\
    -\mat{U}^T & \mat{V}^T
    \end{bmatrix}\>,
\end{equation}
and when we apply this as a similarity transform to the Jordan--Wielandt matrix we get (writing without inverses)
\begin{equation}
    \begin{bmatrix}
    0 & \M_n \\
    \M_n^T & 0
    \end{bmatrix} \begin{bmatrix}
    \mat{U} & -\mat{U}\\
    \mat{V} & \mat{V}
    \end{bmatrix} =   \begin{bmatrix}
    \mat{U} & -\mat{U}\\
    \mat{V} & \mat{V}
    \end{bmatrix}\begin{bmatrix}
    \matg{\Sigma} & 0 \\
    0 & -\matg{\Sigma} 
    \end{bmatrix}\>.
\end{equation}
This reveals the well-known fact that the eigenvalues of the Jordan--Wielandt matrix, which is a nonnegative matrix, are $\pm \sigma_k$ for $1 \le k \le d$, where $d$ is the dimension of the square matrix $\M_n$.  This is a characteristic of adjacency matrices for bipartite graphs: they can always be reordered so that their adjacency matrix is in the above form (not necessarily with square matrix blocks), and the eigenvalues occur in $\pm$ pairs (possibly including $0$).

In particular, here, one largest magnitude eigenvalue of the Jordan--Wielandt matrix is $\sigma_1$, but there is another eigenvalue equally large in magnitude, namely $-\sigma_1$.

% As discussed before, Perron's theorem as extended by Frobenius from positive to  nonnegative \emph{irreducible} matrices says that there may be other simple complex eigenvalues on the circle defining the spectral radius $\rho$. In particular, the period of the matrix plays a role.  As we may see by drawing the bipartite digraph for which the Jordan--Wielandt matrix above is the adjacency matrix, which has $2d$ vertices, there is a circuit containing all vertices and that therefore this matrix is irreducible and hence this extension applies. More, because the graph is bipartite we can see that its period is $h=2$ because the GCD of the lengths of all cycles is $2$. 
%The eigenvector corresponding to $\sigma_1$ can therefore be taken to have only nonnegative components; this is the $2d_n$-dimensional vector formed by joining the first column $u_1$ of $\mat{U}$ to the first column $v_1$ of $\mat{V}$.  

There is more.  We have already shown the computed singular vectors in the case $n=12$ (dimension $d_n=4095$) in Figure~\ref{fig:singularvector12}, although we called them eigenvectors, there (they are: of $\M_n^T\M_n$ or of $\M_n\M_n^T$). The symmetry shown there---namely that the $u$ vector and $v$ vector look to be mirror images of each other---reflects the fact that  $\M_n$ is symmetric about the anti-diagonal; this means that $\M_n^T$ is also symmetric about the anti-diagonal, and hence the Jordan--Wielandt matrix must have eigenvectors symmetric about the half-way point.  Indeed, the vectors $U$ and $V$ are the same, but in reverse order. This suggests that there is an economy that might be useful.  We look for such, in the next section.
% However, we lack some of the tools that we used to attack the fractal eigenvector.  We do not have an asymptotic expansion for the largest singular value, for instance; nor do we have a recurrence relation for the polynomials whose roots are the singular values.

\subsection{Smaller matrices.}

% In contrast to the eigenvector problem, for the \emph{singular} values and singular vectors, we apparently do not have the same advantages.  We seek to understand the visible structures in Figures~\ref{fig:singularvector12} and~\ref{fig:singularvaluesdistribution}. But all we seem to have at the moment is the brute force of computation.

After some thought, we notice that we may use the involutory symmetry of $\M_n$ as follows.  As previously noted, the left and right singular vectors are the same, except in reverse order.  This is because the matrix $\M_n \mat{J}_n$ is \emph{symmetric}, where $\mat{J}_n$ is the involutory ``anti-identity'': for instance, when $n=2$ and the dimension $d_n=2^n-1$ is~$3$, we have
\begin{equation}
    \mat{J}_2 = \begin{bmatrix}
    0 & 0 & 1\\
    0 & 1 & 0 \\
    1 & 0 & 0
    \end{bmatrix}\>.
\end{equation}
From the Jordan--Wielandt matrix, we have that
$\M_n \mat{v} = \sigma\mat{u}$ and $\M_n^T \mat{u} = \sigma\mat{v}$; because $\mat{u} = \mat{J}_n \mat{v}$ and $\mat{v} = \mat{J}_n\mat{u}$, we see that
\begin{equation}
    \M_n \mat{J}_n \mat{u} = \sigma\mat{u}\>.
\end{equation}
That is, a singular vector $\mat{u}$ of $\M_n$ is an \emph{eigenvector} of the sparse, symmetric matrix $\M_n\mat{J}_n$.  Indeed, we have the following propositions.
\begin{proposition}
The matrices $\mat{S}_n = \M_n\mat{J}_n$ can be constructed recursively as follows: as a base, $\mat{S}_1 = [1]$; then
\begin{equation}
    \mat{S}_{n+1} = \begin{bmatrix}
    \mat{e}_1\mat{e}_1^T & 0 & \mat{S}_n\\
    0 & 0 & \mat{e}_1^T \\
    \mat{S}_n & \mat{e}_1 & 0
    \end{bmatrix}\>. \label{eq:srecurrence}
\end{equation}
\end{proposition}
\begin{proof}
From equation~\eqref{eq:Mmatrix_recurrence} we find
\begin{align}
    \M_{n+1}\mat{J}_{n+1} &= \begin{bmatrix}
    \M_n & \mat{0} & \mat{e_1}\mat{e_{d_n}}^T \\
    \mat{e_{d_n}}^T & 0 & \mat{0} \\
    \mat{0} & \mat{e_1} & \M_n 
    \end{bmatrix}\begin{bmatrix}
     & & \mat{J}_n \\
     & 1 & \\
     \mat{J}_n & & 
    \end{bmatrix} \\
    &= \begin{bmatrix}
    \mat{e}_1\mat{e}_1^T & 0 & \M_n\mat{J}_n\\
    0 & 0 & \mat{e}_1^T \\
    \M_n\mat{J}_n & \mat{e}_1 & 0
    \end{bmatrix}
\end{align}
because $\mat{e}_1\mat{e}_{d_n}^T\mat{J}_n = \mat{e}_1\mat{e}_{1}^T$.  Equation~\eqref{eq:srecurrence} follows.
\end{proof}
\begin{proposition}
The matrices $\mat{S}_n = \M_n\mat{J}_n$ have eigenvalues $\pm\sigma$ where $\sigma$ is a singular value of $\M_n$. Moreover, each singular value of $\M_n$ occurs as the absolute value of some eigenvalue of $\mat{S}_n$. 
\end{proposition}
\begin{proof}
Because
\begin{equation}
        \begin{bmatrix}
    0 & \M_n \\
    \M_n^T & 0
    \end{bmatrix} \begin{bmatrix}
    \mat{u} \\
    \mat{v} 
    \end{bmatrix} =  \pm\sigma \begin{bmatrix}
    \mat{u} \\
    \mat{v}
    \end{bmatrix}
\end{equation}
(we do not know which sign of the eigenvalue $\sigma$ belongs to the eigenvector), and because $\mat{u} = \mat{J}_n\mat{v}$, we see that
\begin{equation}
    \M_n\mat{v} = \M_n\mat{J}_n\mat{u} = \pm \sigma \mat{u}\>.
\end{equation}
This establishes that either $\sigma$ or $-\sigma$ is an eigenvalue of $\mat{S}_n = \M_n\mat{J}_n$.  Since $\mat{J}_n$ cannot change the magnitude of the singular values because it is orthogonal, all singular values of $\M_n$ appear as absolute values of eigenvalues of $\mat{S}_n$.
\end{proof}
% Since eigenvalue computation for symmetric matrices is significantly faster than nonsymmetric eigenvalue computation, 
%which is itself faster than singular value computation, 
% RMC Referee 2, Bullet 6, Numerical computation
% this rewriting provides a useful computational technique.  This is 
% especially true because the sparsity pattern of $\mat{S}_n$ is known.
% We found it quite surprising that not all eigenvalues of $\mat{S}_n$ are positive, however; one actually has to take absolute values to get the singular values.

% But in fact more can be done with these than just speed up computation.
Let $D_n(\lambda) = \det(\lambda\mat{I}-\mat{S}_n)$ be the characteristic polynomial of $\mat{S}_n$.
Here are some facts about $\mat{S}_n$, without proof.
\begin{enumerate}
    \item \label{fact:symmetric} $\mat{S}_n$ is symmetric.
    \item $\det \mat{S}_n = -1$ for $n > 1$.
    \item\label{fact:trace} trace$\mat{S}_n = 1$.
    \item If $N_n = $ trace$\mat{S}_n^2$, then $N_{n+1} = 2N_{n}+3$, so $N_n = 2^{n+1}-3$.  Apparently coincidentally, $N_n=2d_n-1$ is the number of nonzero entries in $\M_n$.
    \item $\mat{S}_n = \M_n\mat{J}$ is a matrix square root of $\M_n \M_n^T$.
    \item The eigenvalues of $\mat{S}_n$ are distinct.
\end{enumerate}
Something that isn't quite a ``fact'' is that the digraph of $\mat{S}_n$ looks rather like the directed graph of the Jordan--Wielandt matrix for $\M_{n-1}$, except it has one simple loop on the first vertex.  This is because, apart from that simple loop, the digraph for $\mat{S}_n$ is also bipartite! We will use this in what follows.  Let $\mat{\widetilde{S}}_{n+1}$ be the matrix equal to $\mat{S}_{n+1}$ apart from the $(1,1)$ entry, which is zeroed:  
\begin{align}
    \mat{S}_{n+1} \approx \mat{\widetilde{S}}_{n+1} &= \begin{bmatrix}
    0 & 0 & \mat{S}_n\\
    0 & 0 & \mat{e}_1^T \\
    \mat{S}_n & \mat{e}_1 & 0
     \end{bmatrix}\>.
    % \nonumber\\
    % &= \begin{bmatrix}
    % 0 & \mat{B}\\
    % \mat{B}^T& 0
    % \end{bmatrix}\>,
    \label{eq:bipartiteSpert}
\end{align}
% where $\mat{B}$ is the $(d_{n}+1) \times d_{n}$ matrix consisting of $\mat{S}_n$ stacked on top of $\mat{e}_1^T$.  To make this matrix equal to $\mat{S}_{n+1}$ we must put a $1$ in the $(1,1)$ entry (corresponding to the only loop in the graph).  

% \begin{proposition}\label{thm:sigbound}
% $\sigma_{1,n} \le n$ 
% and 
% $\sigma_{d_n,n} \ge 1/(2n-1)$.
% \end{proposition}
% \begin{proof}
% Since 
% $\|\mat{A}\|_2^2 \le \|\mat{A}\|_1 \|\mat{A}\|_\infty$ 
% by the H\"older inequality,\footnote{One can give an alternative proof based on the maximum degree, namely~$n$, of any vertex in the associated graph. This result from spectral graph theory is well known in general.} and because 
% $\|\mat{S}_n\|_1 = \|\mat{S}_n\|_\infty = n$, 
% we have that 
% $\sigma_{1,n} \le n$.  
% Similarly, because 
% $\|\mat{M}_n^{-1}\|_1=\|\mat{M}_n^{-1}\|_\infty = 2n-1$, 
% we have that 
% $\sigma_{d_n,n} \ge 1/(2n-1)$.
% \end{proof}

% We will also use the graph and adjacency matrix that we get by deleting vertex $2^n$.  Call this matrix $\mat{\widehat{S}}_{n+1}$:
% \begin{equation}
%     \mat{\widehat{S}}_{n+1} = 
%     \begin{bmatrix}
%     \mat{e}_1\mat{e}_1^T & \mat{S}_n \\
%     \mat{S}_n & 0 
%     \end{bmatrix}\>.
% \end{equation}

% , and $\widehat{D}_n(\lambda)$ be the characteristic polynomial of $\mat{\widehat{S}}_n$.

\begin{theorem}\label{thm:alternate}
If $n>1$, then the $d_{n}$ eigenvalues of $\mat{S}_{n}$ arranged in descending magnitude have $d_{n}-1$ alternations in sign, and the largest magnitude eigenvalue is positive. That is, the eigenvalues of $\mat{S}_{n}$ are $(-1)^{i-1}\sigma_{i,n}$ for $1 \le i \le d_{n}$ where $\sigma_{i,n}$ are the singular values of $\M_{n}$. 
\end{theorem}
\begin{remark}
This makes the characteristic
polynomial of $\mat{S}_n$ \emph{self-interlacing} in the sense of~\cite{Tyaglov2017}.
\end{remark}
\begin{proof}
We assume $n\ge 1$ and work with $\mat{S}_{n+1}$.
Consider the graph one gets by deleting vertex $1$, i.e.,~$\mat{\widetilde{S}}_{n+1}$.
Its eigenvalues, which we know are distinct by an additional induction, interlace with those of $\mat{S}_{n+1}$ by Fact~1 of~\cite[Section~47.4]{hogbenhandbook}. Now consider the characteristic polynomial of $\mat{S}_{n+1}$, namely $D_{n+1}(\lambda) = \det(\lambda \mat{I} - \mat{S}_{n+1})$.  Because the determinant is linear in the first row, and the $(1,1)$ entry of the matrix is $\lambda-1$, this is
\begin{align}
D_{n+1}(\lambda) &=    \det \begin{bmatrix}
    \lambda\mat{I}-\mat{e}_1\mat{e}_1^T & 0 & -\mat{S}_n \\
    0 & \lambda & -\mat{e}_1^T \\
    -\mat{S}_n & -\mat{e}_1 & \lambda\mat{I}
    \end{bmatrix} \nonumber\\
&=\det \begin{bmatrix}
    \lambda\mat{I} & 0 & -\mat{S}_n \\
    0 & \lambda & -\mat{e}_1^T \\
    -\mat{S}_n & -\mat{e}_1 & \lambda\mat{I}
    \end{bmatrix}   
    -\det \begin{bmatrix}
    1 &0& 0 & 0\\
    0 &\lambda\mat{I}_{d_n-1} & 0 & -\mat{{L}}_n \\
    0 & 0& \lambda & -\mat{e}_1^T \\
    {-\mat{s}_1} &-\mat{L}_n^T & -\mat{e}_1 & \lambda\mat{I}
    \end{bmatrix}\>.
\end{align}
Here $-\mat{s}_1$ is the first column of $-\mat{S}_n^T$, and $\mat{L_n}$ is defined by the partition $\mat{S}_n^T = [ \mat{s}_1 | \mat{L}_n^T] $; that is, $\mat{L}_n$ is the matrix that remains after we have removed the first row of $\mat{S}_n$.
The first determinant is the characteristic polynomial of the adjacency matrix of a bipartite graph, namely the one that is obtained by deleting the loop in the graph for $\mat{S}_{n+1}$. Its characteristic polynomial, which also can be established by a separate induction to have distinct roots, may therefore be written as $\lambda p(\lambda)p(-\lambda)$, because it is of odd dimension and its nonzero eigenvalues (call them $s_i$, say, for $1 \le i \le 2d_n$) must occur in pairs of positive and negative elements.  The second determinant can be written by Laplace expansion about the first row as $1$ times the characteristic polynomial of the adjacency matrix of another bipartite graph, but now of even dimension; its characteristic polynomial may be written as $q(\lambda)q(-\lambda)$ where the (distinct, by separate induction) roots of this product (call them $t_i$, say, for $1\le i \le 2d_n$) must interlace the eigenvalues of $\mat{S}_{n+1}$ by Fact~1 of~\cite[Section~47.4]{hogbenhandbook}. This means that the signs of $D_{n+1}(t_i) = t_i p(t_i)p(-t_i)$, $i=1, 2, \ldots, 2d_n$ must alternate.  This entails that the $s_i$ also interlace the roots of $\lambda p(\lambda)p(-\lambda)$. Since the nonzero $s_i$ occur in positive and negative pairs, this establishes that there will be $2d_n+1$ sign alternations; and since for a connected graph such as $\mat{S}_{n+1}$ the largest eigenvalue is always positive, the only way to have this interlacing is for the next-largest magnitude eigenvalue be negative.  Similar reasoning establishes that the alternation continues until the requisite intervals are exhausted.  The sign in the smallest interval, which includes $0$, is settled by appealing to the sign of the determinant, which is $-1$.  Because the number of eigenvalues is odd, the smallest eigenvalue is also positive.
\end{proof}
% \begin{remark}
% We have written $D_{n+1}(\lambda)$ as a difference of its odd part and (the negative of its) even part, and interpreted those polynomials as the characteristic polynomials of the graph of $\mat{S}_{n+1}$ with its loop removed and the graph of $\mat{S}_{n+1}$ with its first vertex removed.  Both of those graphs are bipartite.  The conclusion we draw---that the eigenvalues of $\mat{S}_{n+1}$ alternate in sign when sorted in descending order---would thus seem to hold for all graphs on an odd number of vertices which are bipartite except for one loop on vertex~$1$.  We are not aware if this is a known theorem, or if it would have applications elsewhere.
% \end{remark}
\subsection{Patterns versus pareidolia.}
Now we come nearer to our original goal, namely understanding the pictures of the singular vectors.  When we examine the singular vector belonging to $\sigma_{1,k}$ for each of $\M_k$ up to $\M_6$ we begin to see patterns: repeated groups of points shaped vaguely like daggers, or arrowheads, or perhaps boomerangs.  See Figure~\ref{fig:singularvectors67}.
\begin{figure}[h!]
    \centering
    \subcaptionbox{Singular vector $n=6$}{\includegraphics[width=6cm]{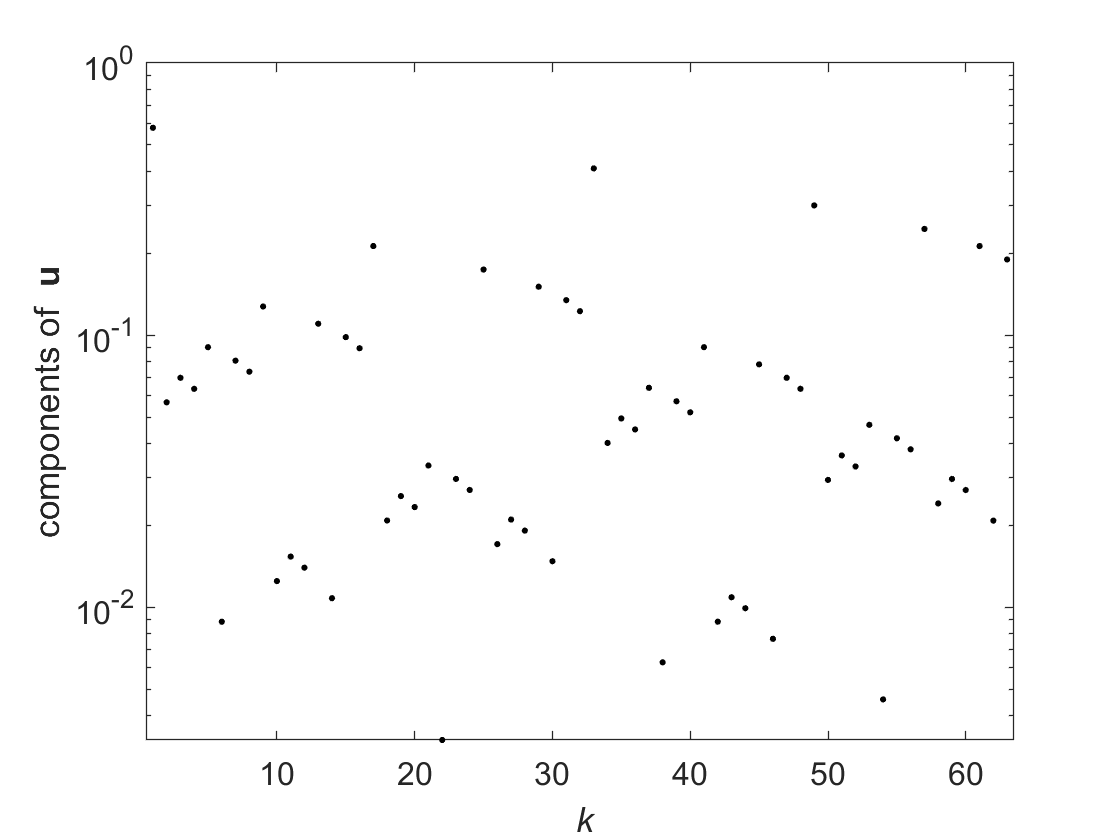}}
    \subcaptionbox{$n=7$}{\includegraphics[width=6cm]{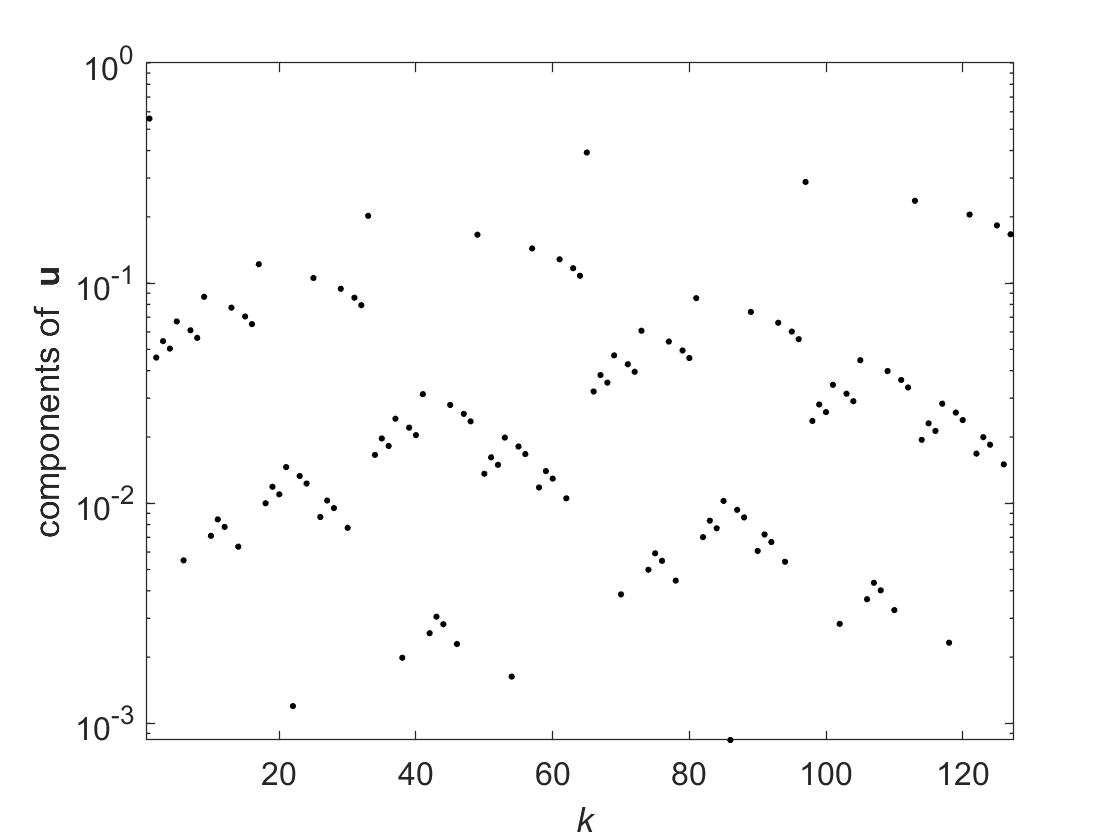}}
    \caption{A discrete plot of the components of the singular vectors corresponding to the dominant singular values of $\M_{6}$ and $\M_7$ drawn on a logarithmic scale. At this density, we begin to see the emergence of complex structures.}
    \label{fig:singularvectors67}
\end{figure}

Moreover, there are the \emph{correct number} of copies of these ``daggers'': twice as many for $\M_{n+1}$ as there were for $\M_n$.  Humans, however, are perhaps overly adroit at seeing patterns---when people see patterns that aren't really there it is called ``pareidolia''---and so we would like to have proof, just as we had for the eigenvector case.  

Let us try do so by a \emph{homotopy continuation}: 
\begin{equation}
    \mat{T}(\e) = \begin{bmatrix}
    \e\mat{e}_1\mat{e}_1^T & 0 & \mat{S}_n\\
    0 & 0 & \e\mat{e}_1^T \\
    \mat{S}_n & \e\mat{e}_1 & 0
    \end{bmatrix}\>. \label{eq:shomotopy}
\end{equation}
When $\e=1$ the eigenvalues of $\mat{T}$ are the same as the singular values of $\M_{n+1}$, in absolute value.  When $\e=0$, the eigenvalues are $0$ and $\{\pm \lambda\}$ where $\{\lambda\}$ are the eigenvalues of $\mat{S}_n$. We just proved in Theorem~\ref{thm:alternate} that the eigenvalues of $\mat{S}_n$ self-interlace, so we may conclude that the set $\{\pm \lambda\}$ contains $2d_n$ distinct values; adding $0$ gives $2d_n+1$ distinct values (no $\M_n$ or $\mat{S}_n$ is singular).  This suggests that we may link the singular values of $\M_{n+1}$ directly to those of $\M_n$ by following the path of each eigenvalue of $\mat{T}$ as $\e$ varies from $0$ to $1$. 

One expects that the largest singular value of $\mat{S}_n$ would be transformed by this process to be the largest singular value of $\mat{S}_{n+1}$.  This seems to work: for every $\e$ the matrix is symmetric and hence the roots are real, and they do not seem to cross (we conjecture that they do not).  
\begin{conjecture}
The eigenvalues of $\mat{T}(\e)$ are simple on $0 \le \e $.
\end{conjecture}
The evidence we have for this conjecture is that the discriminants we have calculated, of the characteristic polynomials of $\mat{T}(\e)$ with respect to $\lambda$, have strictly positive coefficients as monomial basis polynomials in $\e$.  We do not have a general proof.

The largest singular value of the previous matrix becomes the largest singular value of the current one (the new root that starts from the negative of the largest singular value becomes the second largest of the next one).  The plots, put together, are interesting.  See Figure~\ref{fig:homotopyplot}. To compute these, we first computed the characteristic polynomials $F_k(\lambda,\e)$ for $1 \le k \le 5$. We then differentiated $F_k(\lambda(\e),\e)=0$ with respect to $\e$ to get a differential equation for $\lambda(\e)$; this equation is called the \emph{Davidenko} equation~\cite{boyd2014solving}.  For $k=1$, the initial conditions were $\lambda(0) = 0$, $1$, and $-1$.  We used Maple's \texttt{dsolve/numeric}~\cite{shampine2000initial} to solve the Davidenko equation up to $\e=1$.  We then used the solutions at $\e=1$, together with their negatives, and zero, as the initial conditions for a similar problem with $k$ replaced by $k+1$.  We plot the \emph{absolute values} of the eigenvalue paths in Figure~\ref{fig:homotopyplot}.
This gives the appearance of each old singular value giving birth to two new ones; which, in a sense, is true.

% This process also gives, in principle, a homotopy continuation for the singular vectors, but we did not pursue this.

\begin{figure}
    \centering
    \subcaptionbox{Continuation in $\sigma$}
    {\includegraphics[width=5.5cm]{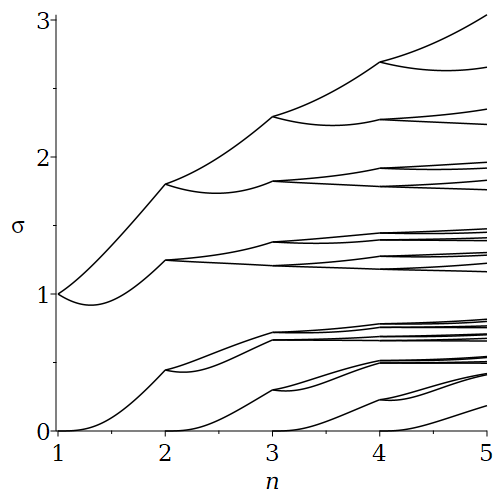}}
    \subcaptionbox{Squared}
    {\includegraphics[width=5.5cm]{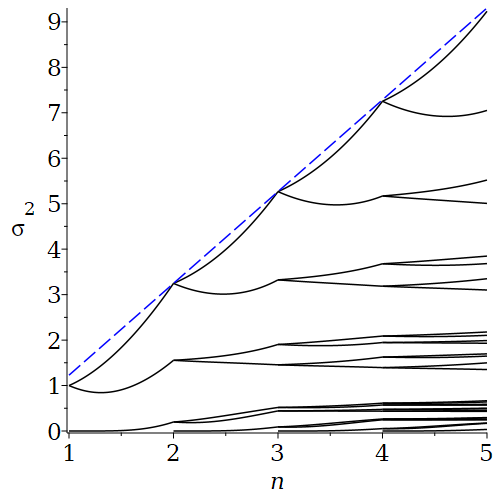}}
    \caption{The homotopies for $n=2$, $3$, $4$, and $5$ plotted together.  The role of $\varepsilon$ is played by $t-1$ in $1\le t \le 2$, by $t-2$ in $2\le t \le 3$, and so on.  Thes figures were drawn by solving the so-called Davidenko equations numerically, where the \emph{initial conditions} for the solution on each interval were provided by the endpoints of the previous solutions, \emph{their negatives}, and the new point $0$. We plot only the absolute values of the eigenvalues on the left, which shows the connections of singular values. On the right, plotting the squares shows our conjectured bound (blue dashed line) from equation~\eqref{eq:conjecturedbound}}
    \label{fig:homotopyplot}
\end{figure}

Another conjecture that comes from this experiment is 
\begin{conjecture}
\begin{equation}
\label{eq:conjecturedbound}
\sigma_{1,n} \le \sqrt{ 2.0193n-0.7914}\>.
\end{equation}
\end{conjecture}
The numbers come from
fitting a straight line to the largest singular values at $n=2$ and at $n=3$.  At $n=4$ the conjectured bound is $0.43$\% larger than the true value. By $n=20$, the conjectured bound is larger than the true value by $0.85$\%. 
%Compare Proposition~\ref{thm:sigbound}.
% {\frac {\sqrt [3]{28+84\,i\sqrt {3}}}{6}}+{\frac {14}{3\,\sqrt [3]{28+
%84\,i\sqrt {3}}}}+{\frac{2}{3}}

We infer from our numerical experiments that there are $2^{n-1}$ double roots larger than $1$ in magnitude at $\e=0$, one root exactly $0$, and $2^{n-1}-1$ double roots smaller than $1$ in magnitude at $\e=0$. As we move to $\e=1$, this changes to $2^{n}$ roots larger than $1$ and $2^n-1$ roots smaller than $1$.  We add a new root at $\sigma=0$ and start again.  This homotopy, if we could prove that it behaves as we think it does, would actually explain Figure~\ref{fig:singularvaluesdistribution} (in particular, it explains the gap near $\sigma = 1$ because the new eigenvalue starting at $0$ never seems to cross the line). 

But in fact we stop here: we now have a partial explanation for Figure~\ref{fig:singularvaluesdistribution}, assuming that our conjecture (that the discriminant has only positive coefficients) is true.

\section{Concluding remarks.}
\begin{quote}
\emph{A construction only becomes interesting when it can
be placed side by side with other analogous constructions
for forming species of the same genus}.
---Henri Poincar\'e~\cite{poincare1905science}
\end{quote}

We showed in Theorem~\ref{thm:blockform} that the dominant eigenvector of $\M_n$ appears to acquire a fractal structure in the limit as $n \to \infty$: the eigenvector has two halves, the bottom half being something related to the previous eigenvector and the top half being a scaled copy of that.  This theorem ``explains'' the visual appearance of the numerically computed eigenvectors; or, at least, one aspect of that appearance.  The presence of powers of $\rho_n$ (for $\rho_n = 2 - O(4^{-n})$) explains the discrete levels of values in the eigenvectors.  The appearance of the OEIS sequence \href{http://oeis.org/A048896}{oeis.org/A048896} at the bottom of the bottom half is explained by the asymptotics of the dominant root, equation~\eqref{eq:asymptrho}.
The (conjectured) appearance of $\pi$ in the upper half of the eigenvector, according to equation~\eqref{eq:conjecturepi}, will need future work to explain.  

We claimed at the beginning of the article that the visible features that we tried to explain were not numerical artifacts, but were in fact faithful to the mathematics; we have not proved that fidelity here. The numerical analysis is almost, but not quite, straightforward.
For a clear treatment of the standard eigenvalue and eigenvector perturbation theory, see~\cite{Greenbaum2020}; for a treatment specialized to perturbation of Perron vectors, see~\cite{Dietzenbacher1988}.
The key fact needed is that the \emph{next} largest eigenvalue is, by the results of~\cite{corless2013largest}, $O(4^{-n})$ away, giving an estimate of $O(d_n^2)$ for the condition number for the dominant eigenvector.
We detailed the argument to our own satisfaction, but it is frankly simpler to do the computation again in high precision (we used Maple's variable precision, with 30 and again with 60 Digits; 30 was more than enough) to estimate the largest relative error in the eigenvector.  Our experiments showed (in agreement with our analysis) that this largest relative error grew like $8\times 10^{-18} d_n^2$, resulting in an error of about $2\times10^{-9}$ by $n=14$ when $d_n = 16383$.  Thus, all our eigenvector plots are correct to better than visual accuracy.   

% RMC Referee 2 Bullet 8. Quantifying our 
% statement, still not proving it.
% Moreover, the patterns in the eigenvector pictures from our numerical computations were further confirmed by our Theorem~\ref{thm:blockform}, so they could not have been the result of numerical error. 

The patterns in the singular values and the singular vectors, on the other hand, were not confirmed by any theorems.  Nonetheless, the results are accurate, because the two largest singular values are not that close to one another, and the gap between them determines the sensitivity of the dominant singular vector~\cite{Greenbaum2020}. 

The Mandelbrot set features prominently in the theory of dynamical systems, and much is known about it.  We suspect that our equation~\eqref{eq:conjecturepi}, which gives an expression for the geometric mean of the nonzero elements of the periodic orbit in the Mandelbrot set corresponding to the largest $c$, must be connected in some way with this vast theory, but at this moment we do not know just in what way.

Mandelbrot matrices and polynomials also have many connections to combinatorial problems.  For instance, the fixed point for equation~\eqref{eq:Mandelbrotpolynomials} is a generating function for the Catalan numbers, so the trailing coefficients of Mandelbrot polynomials are Catalan numbers. The polynomials appear to be \emph{unimodal}, meaning that the coefficients increase in size to a maximum, then decay monotonically.  We know of no proof.

The matrix family $\{\M_n\}$ has been generalized in at least two separate ways.  For instance, we mention the recurrence $q_{n+1} = \lambda q_nq_{n-1} +1$ which generates the \emph{Fibonacci--Mandelbrot} polynomials and their analogous companion matrices which contain only elements $\{-1,0,1\}$; there are puzzles here, too.

The matrices $\M_n$ themselves have yet more to tell us.  The inverse of $\M_n$ is sparse and contains only elements from the population $\{-1,0,1\}$; its largest magnitude eigenvalues correspond to the \emph{smallest} magnitude eigenvalues of $\M_n$.  What more can be said about those eigenvalues and eigenvectors? We look forward to finding out.

\medskip\par\noindent
\emph{Dedicated to the memory of Peter B.~Borwein, May 10, 1953--August 23, 2020.} We remember him by using the ``Peter Borwein end-of-proof symbol'', $\natural$ (``naturally'').
\begin{acknowledgment}{Acknowledgments.}
We thank the reviewers and editors for their thoughtful comments. We also thank the staff at Western Library for their work making the literature accessible. This work was supported by NSERC.
\end{acknowledgment}
%\bibliographystyle{maa.bst}
%\bibliography{mandelbrot}

\end{document}